\documentclass[10pt,a4paper]{article}
\usepackage{amsmath,amssymb,amsthm,esint,bm}
\usepackage{mathrsfs}
\usepackage{bookmark}
\allowdisplaybreaks[3]

\newtheorem{definition}{Definition}[section]
\newtheorem{theorem}[definition]{Theorem}
\newtheorem{lemma}[definition]{Lemma}

\newtheorem{corollary}[definition]{Corollary}
\theoremstyle{remark}
\newtheorem{remark}[definition]{Remark}
\numberwithin{equation}{section}

\title{Gradient  potential estimates  in  elliptic obstacle problems with Orlicz growth}
\author{Qi Xiong \\
School of Mathematical Sciences \\
Nankai University, Tianjin, 300071, China\\
e-mail:\,2120170066@mail.nankai.edu.cn \\
Zhenqiu Zhang\footnote{Z.~Zhang is the corresponding author.}\\
School of Mathematical Sciences and LPMC\\
Nankai University, Tianjin, 300071, China\\
e-mail:\,zqzhang@nankai.edu.cn \\
Lingwei Ma \\
School of Mathematical Sciences \\
Tianjin Normal University, Tianjin, 300387,  China\\
e-mail:\,mlw1103@163.com \\}
\date{\today}

\usepackage{hyperref}
\begin{document}
\maketitle
\begin{abstract}
In this paper,we consider the solutions of the non-homogeneous  elliptic obstacle problems  with Orlicz growth involving  measure data. We first establish the pointwise estimates of  the approximable solutions to these problems via fractional maximal operators. As a result, we obtain pointwise and  oscillation estimates for the gradients of solutions  by the non-linear Wolff potentials, and these  yield results on $C^{1,\alpha}$-regularity of solutions.\\


\end{abstract}


\section{Introduction and main results}\label{section1}

\ \ \  In this paper, we consider  the non-homogeneous  elliptic obstacle problems with Orlicz growth and they are  related to  measure data problems of the type
\begin{equation}\label{1.1}
  \operatorname{div}\left({a}(x,Du)\right) =\mu \quad\quad\mbox{in}\ \ \ \Omega \\[0.05cm]
\end{equation}
where $ \Omega\subseteq \mathbb{R}^n, n\geqslant2 $ is a bounded open set and $\mu$ is a bounded Radon measure  on $\Omega$. Moreover we
assume that $\mu(\mathbb{R}^n \backslash \Omega)=0$ and $a=a(x,\eta): \Omega \times\mathbb{R}^n \rightarrow \mathbb{R}^n$ is measurable for each $x\in \Omega$ and differentiable for almost every $\eta \in \mathbb{R}^n$ and there exist constants $0<v\leqslant 1 \leqslant L<+\infty$ such that for all $x,\eta,\lambda \in \mathbb{R}^n$,
\begin{eqnarray}\label{a(x)1}
  \left\{\begin{array}{r@{}c@{}ll}
&&D_{\eta} a(x,\eta )\lambda \cdot \lambda \geqslant v\dfrac{g(|\eta|)}{|\eta|}|\lambda|^2 \,, \\[0.05cm]
&&|a(x,\eta)|+|\eta||D_{\eta} a(x,\eta )|\leqslant Lg(|\eta|)\,, \\[0.05cm]
  \end{array}\right.
\end{eqnarray}
where $D_{\eta}$ denotes the differentiation in $\eta$ and $g(t) : [0,+\infty)\rightarrow [0,+\infty)$ satisfies
\begin{eqnarray}\label{a(x)3}
  \left\{\begin{array}{r@{}c@{}ll}
&&g(0)=0 \ \ \ \Leftrightarrow \ \ \   t=0 \,, \\[0.05cm]
&&g(\cdot)\in C^{1}(\mathbb{R}^+)\,, \\[0.05cm]
&& 1\leq i_{g}=: \inf_{t>0}\frac{tg'(t)}{g(t)}\leq \sup_{t>0}\frac{tg'(t)}{g(t)}=:s_{g}<\infty. \, \\[0.05cm]
  \end{array}\right.
\end{eqnarray}
We define
\begin{equation}\label{g}
G(t):= \int_0^tg(\tau)\operatorname{d}\!\tau \ \ \ \mbox{for}\ \ t\geq0.
\end{equation}
It is straightforward to see  that  $G(t)$ is convex and  strictly  increasing. The standard example for $G(\cdot)$ is $$G(t)=\int_{0}^{t}(\mu+s^2)^{\frac{p-2}{2}} s ds$$ with $\mu\geqslant0$, then \eqref{a(x)1} is reduced to $p$-growth condition.

 The obstacle condition that we impose on the solution is  of the form $u \geq \psi$ a.e.  on $\Omega$, where $\psi \in W^{1,G}(\Omega)\cap W^{2,1}(\Omega)$ is a given function  and $G$ is defined as \eqref{g}. In the classical setting, we consider an inhomogeneity  $f \in L^{1}(\Omega)\cap (W^{1,G}(\Omega))'$, where $(W^{1,G}(\Omega))'$ is the dual of $W^{1,G}(\Omega)$, the obstacle problem can be formulated by the variational inequality
\begin{equation}\label{fjd}
\int_{\Omega} a(x,Du)\cdot D(v-u)dx \geq \int_{\Omega} f(v-u) dx
\end{equation}
for all functions $v \in u+W_0^{1,G}(\Omega)$ that  satisfy $v \geq \psi $ a.e. on $\Omega$. However, we  are more interested in solutions to obstacle problems with measure data in the sense that we want to replace the inhomogeneity $f$ by a bounded Randon measure $\mu$. And the solutions to the obstacle problems can be obtained by approximation with solutions to variational inequalities \eqref{fjd}.  The definition of approximable solutions is described precisely in Definition \ref{opdy}.

In fact, the nonlinear operator $a(\cdot,\cdot)$ is built upon the model case
\begin{equation*}
  \operatorname{div}\left(\omega(x)\frac{g(|Du|)}{|Du|}Du\right) =\mu \quad\quad\mbox{in}\ \ \ \Omega \\[0.05cm]
\end{equation*}
where $\omega :\Omega\rightarrow[c,+\infty) $ is a bounded measurable and separated from zero function, and $g$ satisfies \eqref{a(x)3}. This type of elliptic equations were first introduced by Lieberman \cite{l1} and moreover he proved $C^{\alpha}$- and $C^{1,\alpha}$-regularity of the solutions for these  elliptic equations in his paper. Since then there has been significant advances in regularity theory for this class of equations, we refer to these article \cite{bm20,cm16,cm17,cm15,l2,rt1}.

 In this work, we are interested in  the connections between regularity properties of the solutions and Wolff potentials of data $\mu$ and $\psi$. The Wolff potential was introduced by Maz'ya and Havin \cite{mh31} and the relevant fundamental contributions were attributed to Hedberg and Wolff \cite{hw32}.  The last years  have seen  important developments in nonlinear potential theory, with a deeper analysis of the interactions between fine properties of Sobolev functions, regularity theory of nonlinear elliptic equations and nonlinear potentials. The fundamental results due to  Kilpel$\ddot{a}$inen Mal$\acute{y}$ \cite{km5,km6}  are the pointwise estimates of solutions to the nonlinear equations of $p$-Laplace type via the   Wolff potentials. Later these results have been extended to a general setting by Trudinger and Wang \cite{tw7,tw8}  by means of  a different  approach. Further  results for the gradient of solutions have been achieved by  Duzaar, Kuusi and Mingione \cite{dm10,dm11,km12,m9}.
 Moreover, Scheven \cite{s26,s28} extended the aboved-mentioned results to elliptic obstacle problems with $p$-growth.  For more results, please see \cite{km13,km15,km14,mz1}.

As for the elliptic equations with Orlicz growth, Baroni \cite{b13}  obtained  pointwise   gradient estimates for solutions of  equations with constant coefficients  by the nonlinear potentials.  Later, these results were upgraded by Xiong and Zhang in \cite{xiong1} to elliptic obstacle problems with measure data. Our goal in this paper is to obtain the pointwise and oscillation estimates for the gradient of solutions  to  obstacle problems with Dini-$BMO$ coefficients. The idea of the proof goes back to Kuusi and  Mingione \cite{km12,m9}. We first derive an excess decay estimates for solutions of obstacle problems by using some comparison estimates. Then iterating the resulting estimates, we give the pointwise estimates of fractional maximal operators. Finally, these estimates allow to draw conclusions about pointwise and oscillation estimates for the gradients of solutions. 

Next, we summarize our main results. We begin by presenting some  definitions, notations and assumptions.
\begin{definition}
A function $B :[0,+\infty)\rightarrow[0,+\infty)$ is called a Young function if it is convex and $B(0)=0$.
\end{definition}
\begin{definition}
Assume that B is a Young function,  the Orlicz class $K^{B}(\Omega)$ is the set of all measurable functions $u : \Omega\rightarrow\mathbb{R}$ satisfying
\begin{equation*}
\int_\Omega B(|u|) \operatorname{d}\!\xi < \infty.\nonumber
\end{equation*}
The Orlicz space $L^{B}(\Omega)$ is the linear hull  of the Orlicz class  $K^{B}(\Omega)$ with the Luxemburg norm
\begin{equation*}
\Vert u \Vert_{L^B(\Omega)}:=\inf\left\lbrace \alpha>0: \ \ \int_{\Omega}B\left(\frac{|u|}{\alpha} \right) \operatorname{d}\!\xi \leqslant1\right\rbrace .
\end{equation*}
Furthermore, the Orlicz-Sobolev space $W^{1,B}(\Omega)$ is defined as
\begin{equation*}
W^{1,B}(\Omega)=\left\lbrace  u\in L^{B}(\Omega)\cap W^{1,1}(\Omega) \ \vert \ Du\in L^{B}(\Omega)\right\rbrace.\nonumber
\end{equation*}
Here, $D$ stands for gradient. The space $W^{1,B}(\Omega)$, equipped with the norm
$\Vert u \Vert_{W^{1,B}(\Omega)}:=\Vert u \Vert_{L^B(\Omega)}+\Vert Du \Vert_{L^B(\Omega)},$ is a Banach space. Clearly, $W^{1,B}(\Omega)=W^{1,p}(\Omega)$, the standard Sobolev space, if $B(t)=t^p$ with $p\geqslant1$.
\end{definition}
Note that for the Luxemburg norm there holds the inequality
\begin{equation*}
\Vert u \Vert_{L^B(\Omega)}\leqslant \int_{\Omega}B(|u|) \operatorname{d}\!\xi +1.
\end{equation*}

 The subspace $W_{0}^{1,B}(\Omega)$ is the closure of $C_{0}^{\infty}(\Omega)$ in $W^{1,B}(\Omega)$. The above properties about Orlicz space can be found in \cite{rr28}.

For every $k>0$ we let
\begin{equation*}
T_{k}(s):=
\left\{\begin{array}{r@{\ \ }c@{\ \ }ll}
s\ \ \ \ \ \ \ \ if\ \ |s|\leqslant k\,, \\[0.05cm]
k\ sgn(s)\ \ \ \ \ \ \ if\ \ |s|> k\,. \\[0.05cm]
\end{array}\right.
\end{equation*}

Moreover, for given Dirichlet boundary data $h\in W^{1,G}(\Omega)$, we define
$$\mathcal{T}^{1,G}_{h}:=\left\lbrace u: \Omega\rightarrow \mathbb{R} \ measurable: T_{k}(u-h)\in W_{0}^{1,G}(\Omega) \ \ for \ all \ k>0\right\rbrace. $$
Next  we introduce the definition of approximable solutions.

\begin{definition}\label{opdy}
Suppose that an obstacle function $\psi \in W^{1,G}(\Omega)$, measure data $\mu \in \mathcal{M}_{b}(\Omega)$ and boundary data $h \in W^{1,G}(\Omega)$ with $h\geq \psi$ a.e. are given. We say that $u \in \mathcal{T}^{1,G}_{h}(\Omega)$ with $u \geq    \psi$ a.e. on $\Omega$ is  a limit of approximating solutions of the obstacle problem $OP(\psi ; \mu)$ if there exist functions
$$f_{i} \in (W^{1,G}(\Omega))'\cap L^{1}(\Omega)\ \  with\ \  f_{i}\stackrel{\ast}\rightharpoonup \mu \ in \ \mathcal{M}_{b}(\Omega) \ \ as \ i\rightarrow+\infty$$
 satisfies
$$\limsup_{i\rightarrow+\infty}\int_{B_R(x_0)}|f_i|dx\leqslant|\mu|(\overline{B_R(x_0)}).$$
and solutions $u_{i}\in W^{1,G}(\Omega)$ with $u_{i}\geqslant \psi$ of the variational inequalities
\begin{equation}\label{opdy1}
\int_{\Omega}a(x,Du_{i})\cdot D(v-u_{i})dx\geqslant \int_{\Omega}f_{i}(v-u_{i})dx
\end{equation}
for $\forall \ v \in u_{i}+W_{0}^{1,G}(\Omega)$ with $v\geqslant \psi$ a.e. on $\Omega$, such that for $i\rightarrow +\infty$,
$$u_{i}\rightarrow u \ \ a.e. \ \ \  on \ \  \Omega$$
and $$u_{i}\rightarrow u \ \ \ in \ \ \ W^{1,1}(\Omega).$$
\end{definition}
For the inequalities \eqref{opdy1} with constant coefficients, the existence of approximating solutions converging in the sense of the above definition has been proved in our preceding work \cite{xiong1}. Then the existence in this paper can be obtained by minor adjustments.

Let us next  turn our attention to  the classical non-linear Wolff potential which is defined by
\begin{equation*}
W^{\mu}_{\beta,p}(x,R):=\int_0^R\left( \frac{|\mu|(B_{\rho}(x))}{\rho^{n-\beta p}}\right) ^{1/(p-1)}\frac{\operatorname{d}\!\rho}{\rho}
\end{equation*}
for parameters $\beta \in (0,n]$  and $p>1$.  We also abbreviate
\begin{equation*}
W^{[\psi]}_{\beta,p}(x,R):=\int_0^R\left( \frac{D\Psi(B_{\rho}(x))}{\rho^{n-\beta p}}\right) ^{1/(p-1)}\frac{\operatorname{d}\!\rho}{\rho}
\end{equation*}
with $D\Psi(B_{\rho}(x)):=\int_{B_{\rho}(x)}\left(\frac{g(|D\psi|)}{|D\psi|}|D^{2}\psi|+1\right)d\xi$.

  Now we recall the definitions of the centered  maximal operators as follows.
\begin{definition}
Let $ \beta\in[0,n], x\in \Omega$ and $ R<dist(x,\partial\Omega) $, and let $u$ be an $ L^1(\Omega) $-function or a measure with finite mass; the restricted fractional $ \beta $ maximal function of  $u$ is defined by
\begin{equation*}
M_{\beta,R}(u)(x):=\sup_{0<r\leqslant R}r^{\beta}\frac{|u|(B_r(x))}{|B_r(x)|}=\sup_{0<r\leqslant R}r^{\beta}\fint_{B_r(x)}|u|\operatorname{d}\!\xi.
\end{equation*}
\end{definition}
Note that when $ \beta=0 $ the one defined above is the classical Hardy-Littlewood maximal operator.

Moreover, we define
\begin{equation*}
\overline{M}_{\beta,R}(\psi)(x):=\sup_{0<r\leqslant R}r^{\beta}\frac{D \Psi(B_r(x))}{|B_r(x)|}=\sup_{0<r\leqslant R}r^{\beta}\fint_{B_r(x)}\left(\frac{g(|D\psi|)}{|D\psi|}|D^{2}\psi|+1\right)\operatorname{d}\!\xi.
\end{equation*}
\begin{definition}
Let $ \beta\in[0,n], x\in \Omega$ and $ R<dist(x,\partial\Omega) $, and let $u$ be an $ L^1(\Omega) $-function or a measure with finite mass; the restricted sharp  fractional $ \beta $ maximal function of  $u$ is defined by
\begin{equation*}
M^{\#}_{\beta,R}(u)(x):=\sup_{0<r\leqslant R}r^{-\beta}\fint_{B_r(x)}|u-(u)_{B_r(x)}|\operatorname{d}\!\xi.
\end{equation*}
\end{definition}
When $ \beta=0 $ the one defined above is the Fefferman-Stein sharp maximal operator.

Throughout this paper we write
$$\theta(a,B_{r}(x_0))(x):=\sup_{\eta \in \mathbb{R}^n\setminus \left\lbrace0\right\rbrace  }\frac{|a(x,\eta)-\overline{a}_{B_{r}(x_0)}(\eta)|}{g(|\eta|)}, $$
where $$\overline{a}_{B_{r}(x_0)}(\eta):=\fint_{B_{r}(x_0)}a(x,\eta)dx.$$
Then we can easily check from \eqref{a(x)1} that $|\theta(a,B_{r}(x_0))|\leqslant2L$. In addition, we assume that $a(x,\eta)$  satisfies the Dini-$BMO$ regularity. More precisely,
\begin{definition}
We say that  $a(x,\eta)$ is ($\delta$, R)-vanishing for some $\delta, R>0$,  if
\begin{equation}\label{a(x)2}
\omega(R):=\sup_{{\substack{ x_{0}\,\in\,\Omega\\0<r\leq R}}} \left( \fint_{B_{r}(x_{0})}\theta(a,B_{r}(x_{0}))^{\gamma'}\operatorname{d}\!x\right) ^{\frac{1}{\gamma'}} \leq\delta,
\end{equation}
where $\gamma'=\frac{\gamma}{\gamma-1}$, $\gamma$ is as in Lemma \ref{nibu}.
\end{definition}

Finally we state our main results of this paper. The first result is  the following Theorem that shows some pointwise estimates of  the approximable solutions to of the non-homogeneous quasilinear elliptic obstacle problems involving measure data via fractional maximal  operators.
\begin{theorem}\label{Th1}
Under the assumption  \eqref{a(x)1},  \eqref{a(x)3} and   \eqref{a(x)2}, let $\psi \in W^{1,G}(\Omega)\cap W^{2,1}(\Omega)$, $\frac{g(|D\psi|)}{|D\psi|}|D^{2}\psi|$ $ \in L^{1}_{loc}(\Omega)$. Assume that $u \in W^{1,1}(\Omega)$ with $u\geqslant \psi$ a.e. is a limit of approximating solutions to $OP(\psi; u)$ with measure data $\mu \in \mathcal{M}_{b}(\Omega)$(in the sense of Definition \ref{opdy}),  and
\begin{equation}\label{dytj}
\sup_{r>0}\int_{0}^{r}[\omega(\rho)]^{\frac{1}{1+s_g}}\frac{d\rho}{\rho}< +\infty,
\end{equation}
then  there exists a constant $c=c(n,i_g,s_g,v,L) $ and a radius $R_{0}>0$, depending on $n,i_g,s_g,v,L,\omega(\cdot)$,  such that
\begin{eqnarray}\label{1.11}
\nonumber &&M_{\alpha,R}^{\#}(u)(x)+M_{1-\alpha,R}(Du)(x)\\ \nonumber
&\leqslant& c\left[ R^{1-\alpha}\fint_{B_R(x)}|Du|d\xi+W_{1-\alpha+\frac{\alpha}{i_g+1},i_g+1}^{\mu}(x,2R)+W_{1-\alpha+\frac{\alpha}{i_g+1},i_g+1}^{[\psi]}(x,2R)\right]  \\
&+&c\int_{0}^{2R}[\omega(\rho)]^{\frac{1}{1+s_g}}G^{-1}\left[\fint_{B_{\rho}(x)}[G(|D\psi|)+G(|\psi|)]d\xi \right]\frac{d\rho}{\rho^{\alpha}}.
\end{eqnarray}
and further assume that
\begin{equation}\label{wtiaojian}
\sup_{r>0}\frac{[\omega(r)]^{\frac{1}{1+s_g}}}{r^{\widehat{\alpha}}}\leqslant c_0,
\end{equation}
for some $\widehat{\alpha}\in[0,\beta)$, then
\begin{eqnarray}\label{1.122}
\nonumber M_{\alpha,R}^{\#}(Du)(x)&\leqslant& c\left\lbrace R^{-\alpha}\fint_{B_R(x)}|Du|\operatorname{d}\! \xi+ \left[  M_{1-\alpha i_g,R}(\mu)(x)\right] ^{\frac{1}{i_g}}+\left[  \overline{M}_{1-\alpha i_g,R}(\psi)(x)\right] ^{\frac{1}{i_g}}\right\rbrace  \\ \nonumber
&+&c\left[  W_{\frac{1}{i_g+1},i_g+1}^{\mu}(x,2R)+W_{ \frac{1}{i_g+1},i_g+1}^{[\psi]}(x,2R)\right]   \\
&+&c\int_{0}^{2R}[\omega(\rho)]^{\frac{1}{1+s_g}}G^{-1}\left[\fint_{B_{\rho}(x)}[G(|D\psi|)+G(|\psi|)]d\xi \right]\frac{d\rho}{\rho^{1+\alpha}}
\end{eqnarray}
holds uniformly in $\alpha\in[0,\widehat{\alpha}]$,
where $c=c(n,i_g,s_g,v,L,\widehat{\alpha},c_0,\omega(\cdot),diam(\Omega))$, $0<R\leqslant \min \left\lbrace R_0,dist(x_0,\partial \Omega)\right\rbrace $ and $\beta$ is as in Lemma \ref{lemma1.6}.
\end{theorem}
Thanks to Theorem \ref{Th1}, we  derive  pointwise and  oscillation estimates for the gradients of solutions to obstacle problems.
\begin{theorem}\label{Th2}
In the same hypothesis of Theorem \ref{Th1}, let $B_{4R}(x_0)\subseteq \Omega, x,y \in B_{\frac{R}{4}}(x_0), 0<R\leqslant\frac{1}{2},$ and for some $\widehat{\alpha}\in[0,\beta)$, assume that
\begin{equation}
c_0:=\sup_{r>0}\int_{0}^{r}\frac{[\omega(\rho)]^{\frac{1}{1+s_g}}}{\rho^{\widehat{\alpha}}}\frac{d\rho}{\rho}<+\infty.
\end{equation}
Then  there exists a constant $c=c(n,i_g,s_g,v, L, c_0,  \widehat{\alpha},\omega(\cdot),diam(\Omega)) $ such that
\begin{eqnarray} \label{du}
\nonumber |Du(x_0)|  &\leqslant& c\left[ \fint_{B_R(x_0)}|Du|d\xi+W_{\frac{1}{i_g+1},i_g+1}^{\mu}(x_0,2R)+W_{ \frac{1}{i_g+1},i_g+1}^{[\psi]}(x_0,2R)\right] \\
&+&c\int_{0}^{2R}[\omega(\rho)]^{\frac{1}{1+s_g}}G^{-1}\left[\fint_{B_{\rho}(x_0)}[G(|D\psi|)+G(|\psi|)]d\xi \right]\frac{d\rho}{\rho}.
\end{eqnarray}
\begin{eqnarray}\label{du-du} \nonumber
&&\vert Du(x)-Du(y) \vert \\ \nonumber
&\leq &c\fint_{B_R(x_0)}\vert Du\vert\operatorname{d}\!\xi\left(\frac{|x-y|}{R} \right) ^{\alpha}  \\ \nonumber
&+& c \left[  W^{\mu}_{-\alpha+\frac{1+\alpha}{1+i_g},i_g+1}(x,2R)+ W^{[\psi]}_{-\alpha+\frac{1+\alpha}{1+i_g},i_g+1}(x,2R)\right]|x-y|^{\alpha} \\ \nonumber
&+&c \left[  W^{\mu}_{-\alpha+\frac{1+\alpha}{1+i_g},i_g+1}(y,2R)+ W^{[\psi]}_{-\alpha+\frac{1+\alpha}{1+i_g},i_g+1}(y,2R)\right]|x-y|^{\alpha} \\ \nonumber
&+&c \left[ \int_{0}^{2R}[\omega(\rho)]^{\frac{1}{1+s_g}}G^{-1}\left[\fint_{B_{\rho}(x)}[G(|D\psi|)+G(|\psi|)]d\xi \right]\frac{d\rho}{\rho^{1+\alpha}} \right]|x-y|^{\alpha} \\
&+&c \left[  \int_{0}^{2R}[\omega(\rho)]^{\frac{1}{1+s_g}}G^{-1}\left[\fint_{B_{\rho}(y)}[G(|D\psi|)+G(|\psi|)]d\xi \right]\frac{d\rho}{\rho^{1+\alpha}}\right]|x-y|^{\alpha},
\end{eqnarray}
holds uniformly in $\alpha\in[0,\widehat{\alpha}]$, where  $\beta$ is as in Lemma \ref{lemma1.6}  and $x,y$ is the Lebesgue's point of $Du$.
\end{theorem}

\begin{remark}
 Our results  extend the results of Scheven \cite{s26}   to  obstacle problems with Orlicz growth. Besides, we merely require a much weaker condition that $a(\cdot,\cdot)$ satisfies the Dini-$BMO$ regularity compared with \cite{s26}.
\end{remark}
The remainder of this paper is organized as follows. Section 2 contains some notions and preliminary results. In Section 3,   we  derive  the  excess decay estimate for solutions to these problems by using some comparison estimates.  In Section 4, we obtain  pointwise and oscillation estimates for the gradients of solutions  by the  Wolff potentials.

\section{Preliminaries}\label{section2}
In this section, we introduce some notions and results which will be used in this paper. Firstly, we denote by $m$ any number in the natural number set $\mathbb{N}$, it is easily verified that
\begin{equation*}
\parallel f-(f)_{\Omega}\parallel_{L^2(\Omega)}=\min_{c\in \mathbb{R}^m}\parallel f-c\parallel_{L^2(\Omega)}
\end{equation*}
for any measurable set $\Omega \subseteq \mathbb{R}^n$ and every functions $f : \Omega \rightarrow \mathbb{R}^m$ such that $f\in L^2(\Omega)$. If $q\in[1,\infty)$, we have
\begin{equation} \label{1.8}
\parallel f-(f)_{\Omega}\parallel_{L^q(\Omega)}\leqslant2\min_{c\in \mathbb{R}^m}\parallel f-c\parallel_{L^q(\Omega)}.
\end{equation}
\begin{definition}
A Young function $B$ is called an $N$-function if
$$0<B(t)<+\infty \ \ for \ t>0$$
and
\begin{equation}\label{nhanshu}
\lim_{t\rightarrow+\infty}\frac{B(t)}{t}=\lim_{t\rightarrow0}\frac{t}{B(t)}=+\infty.
\end{equation}
It's obvious that $G(t)$ is an $N$-function.

The Young conjugate  of a Young function B will be denoted by $B^{\ast}$ and defined as
$$B^{\ast}(t)=\sup_{s\geq 0}\left\lbrace st-B(s)\right\rbrace  \ \ for \ t\geq 0.$$
\end{definition}
In particular,  if $B$ is an $N$-function, then $B^{\ast}$ is  an $N$-function as well.
\begin{definition}
A Young function B is said to satisfy the global $\vartriangle_2$ condition, denoted by $B\in\vartriangle_2$, if there exists a positive constant C such that for every $t>0$,
\begin{equation*}
B(2t)\leq CB(t).
\end{equation*}
Similarly, a Young function B is said to satisfy the global $\bigtriangledown_2$ condition, denoted by $B\in\bigtriangledown_2$, if there exists a  constant $\theta >1$ such that for every $t>0$,
\begin{equation*}
B(t)\leq \frac{B(\theta t)}{2\theta}.
\end{equation*}
\end{definition}

\begin{remark}  \label{remark1}
For an increasing function $f: \mathbb{R}^+\rightarrow\mathbb{R}^+$ satisfying  $\vartriangle_2$ condition $f(2t)\lesssim f(t)$ for $t\geqslant0$, it is easy to prove that $f(t+s)\leqslant c[f(t)+f(s)]$ holds for every $t,s\geqslant0$. Thus $G(t)$ satisfies  $\vartriangle_2$ condition and the subadditivity property: $G(t+s)\leqslant c[G(t)+G(s)]$.
\end{remark}

Next let we recall a basic property of $N$-function, which will be used in the sequel.
\begin{lemma}\cite{yz1}\label{gyoung}
If $B$ is an $N$-function, then $B$ satisfies the following Young's inequality
$$st\leq B^{*}(s)+B(t), \ \ \ for \ \ \forall s,t\geq0.$$
Furthermore, if $B\in \bigtriangleup_{2}\cap \bigtriangledown_{2}$ is an $N$-function, then $B$ satisfies the following Young's inequality with $\forall \varepsilon >0$,
$$st\leq \varepsilon B^{*}(s)+c(\varepsilon)B(t), \ \ \ for \ \ \forall s,t\geq0.$$
\end{lemma}
Note that $G(t)$ satisfies the Young's inequality.

Another important property of Young's conjugate function is the following inequality, which  can be found in \cite{a1}:
\begin{equation}\label{a(x)4}
B^{*}\left( \frac{B(t)}{t}\right) \leqslant B(t).
\end{equation}
\begin{lemma}\cite{cm17,yz1}\label{gwan}
Under the assumption  \eqref{a(x)3}, G(t) is defined in \eqref{g}. Then we have

(1) $G(t)$ is strictly convex $N$-function and
$$G^{\ast}(g(t))\leqslant c G(t) \ \ \ for \ t\geqslant0 \ \ \ and \ \ some\ \  c>0;$$

(2) $G(t) \in\bigtriangledown _{2}.$
\end{lemma}

Note that  G(t) satisfies $\bigtriangleup_2$ and $\bigtriangledown_2$ conditions, then $W^{1,G}(\Omega)$ is a reflexive Banach space, see [\cite{hphp1},Theorem 6.1.4].

\begin{lemma}\cite{cho1,de1}\label{adaog}
Under the assumption  \eqref{a(x)1} and \eqref{a(x)3},  G(t) is defined in \eqref{g}. Then there exists $c=c(n,i_g,s_g,v,L)>0 $ such that
\begin{equation}
[a(x,\eta)-a(x,\xi)]\cdot(\eta-\xi)\geq cG(|\eta-\xi|), \ \ \ \ for \ \ every \ \ x,\eta,\xi \in \mathbb{R}^n.
\end{equation}
Especially, we have
\begin{equation}
a(x,\eta) \cdot \eta \geq cG(|\eta|), \ \ \ \ \ for \ \ every \ \ x,\eta \in \mathbb{R}^n.
\end{equation}
\end{lemma}
The following iteration lemma turns out to be very useful in the sequel.
\begin{lemma}\cite{g11} \label{diedai}
 Let $f(t)$ be a nonnegative function defined on the interval $[a,b]$ with $a \geqslant 0$. Suppose that for $s,t \in [a,b]$ with $t<s$,
 \begin{equation*}
 f(t)\leqslant \frac{A}{(s-t)^{\alpha}}+\frac{B}{(s-t)^{\beta}}+C+\theta f(s)
 \end{equation*}
 holds, where $A,B,C\geqslant0, \alpha,\beta>0$ and $0\leqslant\theta<1$. Then there exists a constant $c=c(\alpha,\theta)$ such that
  \begin{equation*}
 f(\rho)\leqslant c\left( \frac{A}{(R-\rho)^{\alpha}}+\frac{B}{(R-\rho)^{\beta}}+C\right)
 \end{equation*}
 for any $\rho,R \in [a,b]$ with $\rho<R$.
\end{lemma}
The proof of the following lemma can be found in \cite{xiong1}.
\begin{lemma}\label{qianqi}
Let $\Omega \subset \mathbb{R}$ be a bounded domain. Assume that $1+i_g\leqslant n$ , $h \in \mathcal{T}^{1,G}_{0}(\Omega)$ satisfies
\begin{equation*}
\int_{\Omega \cap \left\lbrace |h|\leqslant k\right\rbrace }|Dh|^{1+i_g}dx\leqslant Mk+M^{\frac{1+i_g}{i_g}}
\end{equation*}

 for $\forall \  k>0$,  and fixed constants  $M>0$.
Then we have
$$\int_{\Omega}|h|^{1+\alpha}dx\leqslant c_{1}M^{\frac{1+\alpha}{i_g}},$$
$$\int_{\Omega}|Dh|^{1+\beta}dx\leqslant c_{2}M^{\frac{1+\beta}{i_g}},$$
where $0<\alpha <min\left\lbrace 1,\frac{n(i_g-1)+1+i_g}{n-1-i_g}\right\rbrace, 0<\beta <min\left\lbrace 1,\frac{n(i_g-1)+1}{n-1}\right\rbrace,c_1=c_1(\Omega,n,i_g,\alpha),c_2=c_2(\Omega,n,i_g,\beta).$
\end{lemma}

\section{Comparison estimates }\label{section3}
Our goal in this section is a suitable comparison estimate between  the  solutions of obstacle problems and the solutions of elliptic equations. Therefore we can obtain  the excess decay estimate for solutions of obstacle problems with measure data.
 Let's start with the following lemma, which will be useful in the sequel.
 \begin{lemma}\label{ag}
Assume that $g(t)$ satisfies   \eqref{a(x)3}, $G(t)$ is defined in \eqref{g}. Then we have

(1) for any  $\beta \geq1$,

 \ \ \ \ \ \ \ \ \ $\beta^{i_g}\leq \dfrac{g(\beta t)}{g(t)}\leq \beta^{s_g}$ \ \ \ and \ \ \  $\beta^{1+i_g}\leq \dfrac{G(\beta t)}{G(t)} \leq \beta^{1+s_g}$,  \ \ \  for every $t>0$,

for any  $0<\beta<1$,

 \ \ \ \ \ \ \ \ \ $\beta^{s_g}\leq \dfrac{g(\beta t)}{g(t)}\leq \beta^{i_g}$ \ \ \ and \ \ \
$\beta^{1+s_g}\leq \dfrac{G(\beta t)}{G(t)} \leq \beta^{1+i_g}$, \ \ \ for every $t>0$.

(2) for any  $\beta \geq1$,

 \ \ \ \ \ \ \ \ \ $\beta^{\frac{1}{s_g}}\leq \dfrac{g^{-1}(\beta t)}{g^{-1}(t)}\leq \beta^{\frac{1}{i_g}}$ \ \ \ and \ \ \  $\beta^{\frac{1}{1+s_g}}\leq \dfrac{G^{-1}(\beta t)}{G^{-1}(t)} \leq \beta^{\frac{1}{1+i_g}}$,  \ \ \  for every $t>0$,

for any  $0<\beta<1$,

 \ \ \ \ \ \ \ \ \ $\beta^{\frac{1}{i_g}}\leq \dfrac{g^{-1}(\beta t)}{g^{-1}(t)}\leq \beta^{\frac{1}{s_g}}$ \ \ \ and \ \ \
$\beta^{\frac{1}{1+i_g}}\leq \dfrac{G^{-1}(\beta t)}{G^{-1}(t)} \leq \beta^{\frac{1}{1+s_g}}$, \ \ \ for every $t>0$.
\end{lemma}
\begin{proof}
We first consider the case (1).

For $\beta \geqslant1$,  \eqref{a(x)3} allows us to estimate
\begin{equation*}
\frac{i_g}{\beta t}\leqslant \frac{g'(\beta t)}{g(\beta t)}\leqslant \frac{s_g}{\beta t}.
\end{equation*}
By integrating the above inequality over $[1, \beta]$, we obtain
\begin{equation*}
i_{g} \log {\beta} \leqslant \log \frac{g(\beta t)}{g(t)}\leqslant s_g \log\beta.
\end{equation*}
Hence
\begin{equation*}
\beta^{i_g} \leqslant  \frac{g(\beta t)}{g(t)}\leqslant \beta^{s_g}.
\end{equation*}
We make use of \eqref{a(x)3} again to get
\begin{equation*}
s_{g}\int_{0}^{t}g(s)ds \geqslant \int_{0}^{t}sg'(s)ds=tg(t)-\int_{0}^{t}g(s)ds,
\end{equation*}
which implies that
\begin{equation*}
tg(t)\leqslant (s_g+1)G(t),
\end{equation*}
then we have
\begin{equation*}
\left(\log G(t) \right) '\leqslant (s_g+1)\left( \log t\right) '.
\end{equation*}
By integrating the above inequality over $[t, \beta t] $ for any $t>0$, we conclude that
\begin{equation*}
\log \frac{G(\beta t)}{G(t)}=\int_{t}^{\beta t}\left( \log G(s)\right) 'ds\leqslant \int_{t}^{\beta t}(s_g+1)\left( \log s\right)'ds=(s_g+1)\log \beta.
\end{equation*}
Thus
\begin{equation*}
G(\beta t)\leqslant \beta ^{s_g+1}G(t).
\end{equation*}
Likewise,
\begin{equation*}
G(\beta t)\geqslant \beta ^{i_g+1}G(t).
\end{equation*}
For $0<\beta<1$, we can prove it in the same way.

As for the case (2), being  $g$ strictly increasing and with infinite limit, then $g^{-1}$ exists, is defined for all $t\in \mathbb{R}$ and it is strictly increasing. Because of (1), for $\beta \geqslant1$, we have
$$g(\beta t)\leqslant \beta^{s_g}g(t).$$
Consequently,
$$\beta t \leqslant g^{-1}(\beta^{s_g}g(t)).$$
Now we choose $t=g^{-1}(s)$, then we get
$$\beta g^{-1}(s) \leqslant g^{-1}(\beta^{s_g}s),$$
then
$$\beta^{\frac{1}{s_g}} g^{-1}(s) \leqslant g^{-1}(\beta s).$$
In other cases we can prove it similarly.

\end{proof}

It's obvious that  lemma \ref{ag} implies that
\begin{equation}\label{lg}
L^{1+s_g}(\Omega)\subset L^{G}(\Omega) \subset L^{1+i_g}(\Omega) \subset L^1(\Omega).
\end{equation}

The next lemma provide us with the  comparison estimate between the solutions of elliptic obstacle problems with measure data and the solutions of the corresponding homogeneous obstacle problems.
\begin{lemma}\label{bijiao}
Assume that conditions  \eqref{a(x)1}-\eqref{a(x)3} are fulfilled, let $B_{2R}(x_0)\subset \Omega,  f \in L^{1}(B_R(x_0))\cap (W^{1,G}(B_R(x_0)))'$ and the map $u\in W^{1,G}(B_R(x_0))$ with $u \geq \psi$ solves the variational inequality
\begin{equation}\label{u0}
\int_{B_R(x_0)} a(x,Du)\cdot D(v-u)dx \geq \int_{B_R(x_0)} f(v-u) dx
\end{equation}
for any  $v \in u+W_0^{1,G}(B_R(x_0))$ that  satisfy $v \geq \psi $ a.e.  on $B_R(x_0)$.
Let $w \in u+W_{0}^{1,G}(B_R(x_0))$ with $w \geq \psi $ be the weak solution of the homogeneous obstacle problem
\begin{equation}\label{w-v}
\int_{B_R(x_0)} a(x,Dw)\cdot D(v-w)dx \geq 0.
\end{equation}
Then there exists $c=c(n,i_g,s_g,v,L)$ such that
\begin{equation}
\fint_{B_R(x_0)} |Du-Dw| \operatorname{d}x\leqslant c\left[R \fint_{B_R(x_0)}|f|dx\right] ^{\frac{1}{i_g}}.
\end{equation}
\end{lemma}

\begin{proof}
Without loss of generality we may assume that $x_0=0, R=1$ by defining

$$\widehat{u}(x)=\dfrac{u(Rx+x_0)}{R}, \ \ \ \widehat{w}(x)=\dfrac{w(Rx+x_0)}{R}, \ \ \ \widehat{f}(x)=Rf(Rx+x_0).$$

Case 1: $\int_{B_1}|f|dx\leqslant 1$. If $1+i_g>n$, because of $u-w\in W_{0}^{1,G}(B_1)$, $u-w\in W_{0}^{1,1+i_g}(B_1)$, then we make use of Sobolev's inequality to get $u-w\in L^{\infty}(B_{1})$. Now we take $v=\frac{u+w}{2}\in u+W_{0}^{1,G}(B_1)$ as comparison functions in the variational inequalities \eqref{u0} and \eqref{w-v}, which implies that
\begin{eqnarray*}
\int_{B_1}|Du-Dw|^{1+i_g}dx &\leq & c\int_{B_1}[G(|Du-Dw|)+1]dx    \nonumber \\
&\leq & c\int_{B_1}[a(x,Dw)-a(x,Du)] \cdot (Dw-Du)dx +c \nonumber  \\
&\leq & c\int_{B_1}f(w-u)dx+c    \nonumber  \\
&\leq & c\parallel u-w\parallel_{L^{\infty}(B_1)}\int_{B_1}|f|dx+c \\
&\leq & c \parallel Du-Dw\parallel_{L^{1+i_g}(B_1)}+c,
\end{eqnarray*}
where we used Lemma \ref{ag} and Lemma \ref{adaog}.
Then we have
\begin{equation*}
\int_{B_1}|Du-Dw|dx\leqslant c.
\end{equation*}
If $1+i_g\leqslant n$, we define
$$D_{k}:=\left\lbrace x \in B_{1} : |u(x)-w(x)| \leq k \right\rbrace, \ \ \ \forall \ k>0, $$
and let $v_k:=u+T_{k}(w-u)$ and $\overline{v_k}:=w-T_{k}(w-u) \in u+W_{0}^{1,G}(B_1)$ as comparison functions in the variational inequalities \eqref{u0} and \eqref{w-v} respectively, then  we have
 \begin{equation}\label{k-}
 \int_{B_1}[a(x,Dw)-a(x,Du)] \cdot D[T_{k}(w-u)]dx \leq k.
 \end{equation}
 Then  for every $k\geqslant1$, we have
\begin{eqnarray*}
\int_{D_k}|Du-Dw|^{1+i_g}dx
&\leq & c\int_{B_1}[a(x,Dw)-a(x,Du)] \cdot D[T_{k}(w-u)]dx+c    \nonumber  \\
&\leq & ck.
\end{eqnarray*}
 Now using Lemma \ref{qianqi}, we obtain
$$\int_{B_1}|Du-Dw|dx \leqslant c$$

Case 2: $\int_{B_1}|f|dx>1$.
We let $$\overline{u}(x)=A^{-1}u(x), \ \ \ \overline{v}(x)=A^{-1}v(x), \ \ \ \overline{w}(x)=A^{-1}w(x),$$
$$\overline{\psi}(x)=A^{-1}\psi(x), \ \ \ \overline{f}(x)=A^{-i_g}f(x), \ \ \ \overline{a}(x,\eta)=A^{-i_g}a(x,A\eta), \ \ \ \overline{G}(t)=\int_{0}^{t}\overline{g}(\tau)d\tau,$$
where $A=(\int_{B_1}|f|dx)^{\frac{1}{i_g}}>1$.
Then we can easily obtain $\int_{B_1}|\overline{f}|dx=1$, $\overline{a}$ satisfies \eqref{a(x)1} and $\overline{g}$ satisfies \eqref{a(x)3}.

Moreover,
$\overline{u}\in W^{1,G}(B_1)$ with $\overline{u} \geq \overline{\psi}$ solves the variational inequality
\begin{equation*}
\int_{B_1} \overline{a}(x,D\overline{u})\cdot D(\overline{v}-\overline{u})dx \geq \int_{B_1} (\overline{v}-\overline{u}) \overline{f}dx
\end{equation*}
for any  $\overline{v} \in \overline{u}+W_0^{1,G}(B_1)$ that  satisfy $\overline{v} \geq \overline{\psi} $ a.e.  on $B_1$.

And
$\overline{w}\in \overline{u}+W_{0}^{1,G}(B_1)$ with $\overline{w} \geq \overline{\psi}$ solves the  inequality
\begin{equation*}
\int_{B_1} \overline{a}(x,D\overline{w})\cdot D(\overline{v}-\overline{u})dx \geq 0.
\end{equation*}

Similar to the case 1 we have
\begin{equation*}
\int_{B_1} |D\overline{u}-D\overline{w}|dx \leqslant c.
\end{equation*}
In conclusion, we have
\begin{equation*}
\int_{B_1}|Du-Dw|dx \leqslant c(\int_{B_1}|f|dx)^{\frac{1}{i_g}},
\end{equation*}
which finishes our proof.
\end{proof}

\begin{corollary}\label{coro}
Assume that conditions  \eqref{a(x)1}-\eqref{a(x)3} are fulfilled, let w be as in Lemma \ref{bijiao} and $\mu \in  \mathcal{M}_{b}(\Omega)$ and u be a limit of approximating solutions for $OP(\psi, \mu)$, in the sense of Definition \ref{opdy}. Then there exists $c=c(n,i_g,s_g,v,L)$ such that
\begin{equation} \label{coro1}
\fint_{B_R(x_0)}|Du-Dw|dx\leqslant c\left[\frac{|\mu|(\overline{B_R}(x_0))}{R^{n-1}}\right]^{\frac{1}{i_g}}.
\end{equation}
\end{corollary}
The proof of corollary \ref{coro} is similar to the proof corollary 4.2 in \cite{xiong1}, so we don't repeat it here.

Next, we prove the Caccioppoli's inequality for homogeneous obstacle problems.
\begin{lemma}\label{caccio}
Assume that conditions  \eqref{a(x)1}-\eqref{a(x)3} are fulfilled, let $B_{R}(x_0)\subseteq \Omega, \psi \in W^{1,G}(B_{R}(x_0))$ and $u \in W^{1,G}(B_R(x_0)) $ with $u\geqslant \psi$ solves the inequality
\begin{equation}\label{vu0}
\int_{B_R(x_0)}a(x,Du)\cdot D(v-u)dx\geqslant 0
\end{equation}
for any $v \in u+W_0^{1,G}(B_R(x_0))$ with $v\geqslant \psi \ a.e. \ on \ B_{R}(x_0)$.
Then there exists $c=c(n,i_g,s_g,v,L)>0$ such that
\begin{equation}
\fint_{B_{\frac{R}{2}}(x_0)}G(|Du|)dx\leqslant c\fint_{B_R(x_0)}G\left( \frac{|u-\lambda|}{R}\right) dx+c\fint_{B_R(x_0)}\left[ G\left( \frac{|\psi|}{R}\right) +G(|D\psi|)\right] dx
\end{equation}
for every $\lambda \geqslant0.$
\end{lemma}
\begin{proof}
Without loss of generality we may assume that $x_0=0$, let $\eta \in C_{0}^{\infty}(B_s)$,  $\eta=1$ on $B_{t}$,  $|D\eta|\leqslant \frac{c}{s-t}$, $0\leqslant\eta \leqslant1$, \ for any $0<t<s\leqslant R$. We take $v=u-\eta (u-\lambda)+\eta \psi \geqslant \psi$ a.e. on  $B_{R}$ as testing function for the inequality \eqref{vu0}, then we have
\begin{equation*}
\int_{B_{s}}a(x,Du) \cdot D(-\eta (u-\lambda)+\eta \psi)dx \geqslant 0,
\end{equation*}
which implies that
\begin{equation*}
\int_{B_{s}}a(x,Du) \cdot (Du) \eta dx \leqslant -\int_{B_s}a(x,Du)\cdot (D\eta) (u-\lambda) dx+\int_{B_s}a(x,Du)\cdot D (\eta\psi)dx.
\end{equation*}
Then combing  Young's inequality with Lemma \ref{ag}, Lemma \ref{gwan}, Lemma \ref{adaog}, we have
\begin{eqnarray*}
\int_{B_t}G(|Du|)dx&\leqslant& \int_{B_s}G(|Du|)\eta dx \\
&\leqslant& c\int_{B_s}a(x,Du)\cdot (Du)\eta dx \\
&\leqslant& c\int_{B_s}|a(x,Du)|| D\eta||u-\lambda| dx+c\int_{B_s}|a(x,Du)||D(\eta \psi)| dx \\
&\leqslant& c \int_{B_s}g(|Du|)\frac{|u-\lambda|}{s-t}dx+c\int_{B_s}g(|Du|)|D\psi|dx+c \int_{B_s}g(|Du|)\frac{|\psi|}{s-t}dx \\
&\leqslant& c\varepsilon \int_{B_s}G^{\ast}(g(|Du|))dx+c(\varepsilon)\int_{B_s}G(|D\psi|)dx
 \\
&+&c(\varepsilon)\int_{B_s}G\left( \frac{|u-\lambda|}{s-t}\right)+G\left( \frac{|\psi|}{s-t}\right)dx \\
&\leqslant& c\varepsilon \int_{B_s}G(|Du|)dx+c(\varepsilon)\int_{B_s}G(|D\psi|)dx
\\&+&c(\varepsilon)\int_{B_s}\left(\frac{R}{s-t} \right) ^{c_1}G\left( \frac{|u-\lambda|}{R}\right)+\left(\frac{R}{s-t} \right) ^{c_1}G\left( \frac{|\psi|}{R}\right)dx.
\end{eqnarray*}
Now we take $\varepsilon$ small enough to get $c\varepsilon \leqslant \frac{1}{2}$, then we make use of Lemma \ref{diedai}, we obtain
\begin{equation}
\fint_{B_{\frac{R}{2}}}G(|Du|)dx\leqslant c\fint_{B_R}\left[ G\left( \frac{|u-\lambda|}{R}\right)+G\left( \frac{|\psi|}{R}\right)\right]  dx+c\fint_{B_{R}}G(|D\psi|)dx.
\end{equation}
\end{proof}
To obtain Lemma \ref{reverse}, we need a new Sobolev type inequality that can be found in \cite{cf1}.
\begin{lemma}\label{zhongjie}
Let $G(t)$ is defined in \eqref{g}. Set

\begin{center}
$S(t):=G(t)\left[\frac{G(t)}{t} \right] ^{-\frac{1}{n}}$ \ \ \ \ for \ \ $t>0$.
\end{center}
Then there exists a constant c depending only on n such that
\begin{equation*}
G^{-1}\left( \fint_{B_R(x)}G\left( \frac{|u-m(u)|}{R}\right) d\xi\right) \leqslant cS^{-1}\left( \fint_{B_R(x)}S(c|Du|)d\xi \right)
\end{equation*}
for any $B_R(x)\subset \mathbb{R}^{n}$ and every weakly differentiable function $u: B_R(x)\rightarrow \mathbb{R}$, where
\begin{center}
$m(u):=\sup \left\lbrace t\in\mathbb{R} :|{y\in B_R(x)  : u(y)>t}| >\frac{|B_R(x)|}{2}\right\rbrace, $
\end{center}
the largest median of u, and $|\cdot|$ denotes Lebesgue measure.
\end{lemma}
Next, we give the proof of the Reverse H$\ddot{o}$lder's inequality to deduce Lemma \ref{diyi}.
\begin{lemma}\label{reverse}
Under the  conditions  \eqref{a(x)1}-\eqref{a(x)3}, we assume that $B_R(x_0)\subseteq \Omega$, $u\geqslant 0, \psi \in W^{1,G}(B_{R}(x_0))$ and $u \in W^{1,G}(B_R(x_0))$ with $u\geqslant \psi$ solves the inequality
\begin{equation*}
\int_{B_R(x_0)}a(x,Du)\cdot D(v-u)dx\geqslant 0
\end{equation*}
for any $v \in u+W_0^{1,G}(B_R(x_0))$ with $v\geqslant \psi \ a.e. \ on \ B_{R}(x_0)$.
Then there exists $c=c(n,i_g,s_g,v,L)>0$ such that
\begin{equation}
\fint_{B_{\frac{3R}{4}}(x_0)}G(|Du|)dx\leqslant cG\left( \fint_{B_R(x_0)}|Du|dx\right)+c\fint_{B_{R}(x_0)}\left[ G(|D\psi|)+G(|\psi|)\right] dx.
\end{equation}
\end{lemma}
\begin{proof}
Without loss of generality we may assume that $x_0=0, R=1$ by defining
\begin{center}
$\bar{u}(x):=\frac{u(x_0+Rx)}{R}$,\ \ \ $\bar{v}(x):=\frac{v(x_0+Rx)}{R}.$
\end{center}
Now we let
\begin{center}
$S(t):=G(t)\left[\frac{G(t)}{t} \right] ^{-\frac{1}{n}}$ \ \ \ \ for \ \ $t>0$.
\end{center}
From Remark \ref{remark1}, Lemma \ref{caccio} and Lemma \ref{zhongjie}, we obtain
\begin{equation}\label{guodu}
\fint_{B_{\frac{\rho}{2}}(y)}G(|Du|)dx \leqslant c (G \circ S^{-1})\left( \fint_{B_{\rho}(y)}S(|Du|)dx\right) +c\fint_{B_{\rho}(y)}G(|D\psi|)dx+c\fint_{B_{\rho}(y)}G\left(\frac{|\psi|}{\rho}\right)dx.
\end{equation}
Now we let $r\leqslant1$, $\alpha \in (0,1)$ and a point $y \in B_{\alpha r}:=B_{\alpha r}(0)$. We take $\rho=(1-\alpha)r$, note that $B_{(1-\alpha)r}(y)\subset B_{r}$, from \eqref{guodu} we get
\begin{eqnarray*}
\fint_{B_{\frac{(1-\alpha)r}{2}}(y)}G(|Du|)dx &\leqslant& c (G \circ S^{-1})\left( \fint_{B_{(1-\alpha)r}(y)}S(|Du|)dx\right) +c\fint_{B_{(1-\alpha)r}(y)}G(|D\psi|)dx \\
&+&\frac{c}{[(1-\alpha)r]^{1+s_g}}\fint_{B_{(1-\alpha)r}(y)}G(|\psi|)dx.
\end{eqnarray*}
On the other hand, thanks to H$\ddot{o}$lder's inequality, we have
\begin{eqnarray}\label{guodu2}
\fint_{B_{(1-\alpha)r}(y)}S(|Du|)dx&=& \nonumber \fint_{B_{(1-\alpha)r}(y)}[G(|Du|)]^{\frac{n-1}{n}}|Du|^{\frac{1}{n}}dx \\
&\leqslant& \left( \fint_{B_{(1-\alpha)r}(y)}G(|Du|)dx\right) ^{\frac{n-1}{n}}\left( \fint_{B_{(1-\alpha)r}(y)}|Du|dx\right) ^{\frac{1}{n}}.
\end{eqnarray}
Now we consider a Young function  $E(t):=S(t^{n})$ and  its Young conjugate function   $E^{\ast}$. Firstly,  Obviously E(t) is increasing and satisfies \eqref{nhanshu} and $E(0)=0$. Then by calculating, we know that
 $$2n-1\leqslant \frac{tE'(t)}{E(t)}\leqslant (n-1)(s_g+1)+1,$$
whence
 $$\widehat{E}(t):=\int_{0}^{t}\frac{E(s)}{s}ds\simeq E(t)$$
and $\widehat{E}$ is convex, so we  can suppose E convex. And it is easy to  see that E(t) satisfies $\bigtriangleup_{2}$ and $\bigtriangledown_{2}$ conditions. In conclude, E(t) satisfies Young's inequality. We change variable $(s=\sigma^{\frac{1}{n}})$  in the definition of the Young's conjugate function, for $\alpha>0$, $\alpha^{\frac{n-1}{n}}s\geqslant S(s^{n})$,
\begin{equation}\label{exing}
E^{*}(\alpha^{\frac{n-1}{n}})=\sup_{s>0}\alpha^{\frac{n-1}{n}}s-S(s^{n})\leqslant c(n)[\sup_{\sigma>0}\alpha^{n-1}\sigma-[S(\sigma)]^n]^{\frac{1}{n}}:=[F^{*}(\alpha^{n-1})]^{\frac{1}{n}},
\end{equation}
where $F(t):=[S(t)]^{n}$. From \eqref{a(x)4}, we deduce that
\begin{equation*}
F^{*}([G(\tau)]^{n-1})=F^{*}\left( \frac{[S(\tau)]^n}{\tau}\right)=F^{*}\left(\frac{F(\tau)}{\tau} \right)\leqslant F(\tau)=[S(\tau)]^n.
\end{equation*}
Next, we take $\tau=G^{-1}(\alpha)$, then we have
\begin{equation}\label{fwan}
F^{*}(\alpha^{n-1})\leqslant [S(G^{-1}(\alpha))]^{n}.
\end{equation}
Coupling \eqref{exing} with \eqref{fwan} tell us  that for $\alpha=\fint_{B_{(1-\alpha)r}(y)}G(|Du|)dx$,
\begin{equation}\label{fwan2}
E^{*}\left(\left( \fint_{B_{(1-\alpha)r}(y)}G(|Du|)dx\right) ^{\frac{n-1}{n}} \right) \leqslant c(n) (S\circ G^{-1})\left(\fint_{B_{(1-\alpha)r}(y)}G(|Du|)dx \right).
\end{equation}
Note that $t\mapsto (G \circ S^{-1})(t)$ is increasing, then
$$(G \circ S^{-1})(t+s)\leqslant c[(G \circ S^{-1})(s)+(G \circ S^{-1})(t)] \ \ \ \ \ \ \ \ \ \  for \ \ \ s,t\geqslant 0.$$
Thanks to Lemma \ref{ag},  \eqref{guodu2},  \eqref{fwan2} and Young's inequality, we get
\begin{eqnarray*}
\fint_{B_{\frac{(1-\alpha)r}{2}}(y)}G(|Du|)dx &\leqslant& c (G \circ S^{-1})\left( \fint_{B_{(1-\alpha)r}(y)}S(|Du|)dx\right) +c\fint_{B_{(1-\alpha)r}(y)}G(|D\psi|)dx \\
&+&c[(1-\alpha)r]^{-1-s_g}\fint_{B_{(1-\alpha)r}(y)}G(|\psi|)dx \\
&\leqslant& c(G \circ S^{-1})\left[ \left( \fint_{B_{(1-\alpha)r}(y)}G(|Du|)dx\right) ^{\frac{n-1}{n}}\left( \fint_{B_{(1-\alpha)r}(y)}|Du|dx\right) ^{\frac{1}{n}}  \right] \\
&+&c\fint_{B_{(1-\alpha)r}(y)}G(|D\psi|)dx+c[(1-\alpha)r]^{-1-s_g}\fint_{B_{(1-\alpha)r}(y)}G(|\psi|)dx \\
&\leqslant& c(G \circ S^{-1})\left\lbrace \varepsilon E^{*}\left[ \left( \fint_{B_{(1-\alpha)r}(y)}G(|Du|)dx\right) ^{\frac{n-1}{n}}\right]\right\rbrace \\
&+& c(G \circ S^{-1})\left\lbrace c(\varepsilon) E \left[\left( \fint_{B_{(1-\alpha)r}(y)}|Du|dx\right) ^{\frac{1}{n}}  \right]  \right\rbrace \\
&+&c\fint_{B_{(1-\alpha)r}(y)}G(|D\psi|)dx+c[(1-\alpha)r]^{-1-s_g}\fint_{B_{(1-\alpha)r}(y)}G(|\psi|)dx \\
&\leqslant& c \varepsilon^{c_1} \fint_{B_{(1-\alpha)r}(y)}G(|Du|)dx+c G\left( \fint_{B_{(1-\alpha)r}(y)}|Du|dx \right)  \\
&+&c\fint_{B_{(1-\alpha)r}(y)}G(|D\psi|)dx+c[(1-\alpha)r]^{-1-s_g}\fint_{B_{(1-\alpha)r}(y)}G(|\psi|)dx \\
&\leqslant& c \varepsilon^{c_1} \fint_{B_{(1-\alpha)r}(y)}G(|Du|)dx+c[(1-\alpha)r]^{-n(1+s_g)}G\left( \int_{B_{(1-\alpha)r}(y)}|Du|dx \right)  \\
&+&c\fint_{B_{(1-\alpha)r}(y)}G(|D\psi|)dx+c[(1-\alpha)r]^{-1-s_g}\fint_{B_{(1-\alpha)r}(y)}G(|\psi|)dx. \\
\end{eqnarray*}
Since $y \in B_{\alpha r}$, the ball $B_{\alpha r}$ can be covered by some balls included in $B_r$ such that only a finite and independent of $\alpha$ number of balls of double radius intersect,  then we have
\begin{eqnarray*}
\int_{B_{\alpha r}}G(|Du|)dx&\leqslant& c \varepsilon^{c_1}\int_{B_{r}}G(|Du|)dx+c[(1-\alpha)r]^{-ns_g}G\left( \int_{B_1}|Du|dx \right) \\
&+&c\int_{B_{1}}G(|D\psi|)dx+c[(1-\alpha)r]^{-1-s_g}\int_{B_1}G(|\psi|)dx.
\end{eqnarray*}
We make use of Lemma \ref{diedai} to get
\begin{equation*}
\int_{B_{\frac{3}{4}}}G(|Du|)dx\leqslant c G\left( \int_{B_1}|Du|dx\right) +c\int_{B_{1}}[G(|D\psi|)+G(|\psi|)]dx,
\end{equation*}
which finishes our proof.
\end{proof}
In order to prove Lemma \ref{diyi}, we introduce the following Lemma \ref{nibu}. Since we have shown the Caccioppoli's inequality (that is, Lemma \ref{caccio}), then we can get Lemma \ref{nibu} by combining [\cite{de1}, Theorem 7] and [\cite{de1}, Proposition 6], for more comments see [\cite{de1}, Theorem 9].
\begin{lemma}\label{nibu}
We assume that $u\in W^{1,G}(B_{R}(x_0))$ solves the inequality \eqref{vw1}, then we have
\begin{equation*}
\left( \fint_{B_{R}(x_0)}G(|Du|)^{\gamma}dx\right) ^{\frac{1}{\gamma}}\leqslant c\fint_{B_{\frac{3R}{2}}(x_0)}G(|Du|)dx
\end{equation*}
for some $c>0$ and $\gamma>1$, depending only on $n, i_g, s_g, v$ and $L$.
\end{lemma}
Now we prove the comparison estimate between the solutions of a homogeneous obstacle problem and the solutions of a desired obstacle problem, and it  will be crucial for dealing with Dini-continuous vector fields $a$.
\begin{lemma} \label{diyi}
Under the conditions  \eqref{a(x)1}-\eqref{a(x)3}, we assume that $B_{2R}(x_0)\subseteq \Omega$, $u\geqslant 0, \psi \in W^{1,G}(B_{2R}(x_0))$, and  $u \in W^{1,G}(B_{2R}(x_0))$ with $u\geqslant \psi$ solves the inequality
\begin{equation}\label{vw1}
\int_{B_R(x_0)}a(x,Du)\cdot D(v-u)dx\geqslant 0
\end{equation}
for any $v \in u+W_0^{1,G}(B_R(x_0))$ with $v\geqslant \psi \ a.e. \ on \ B_{R}(x_0)$. Assume that  $w \in W^{1,G}(B_{2R}(x_0))$ with $w \geqslant \psi$ solves the inequality
\begin{equation}\label{vw2}
\int_{B_R(x_0)}\overline{a}_{B_R(x_0)}(Dw)\cdot D(v-w)dx\geqslant 0
\end{equation}
and $w=u$ on $\partial B_{R}(x_0)$.
Then we have
\begin{equation*}
\fint_{B_R(x_0)}|Du-Dw|dx \leqslant  c \omega(R)^{\frac{1}{1+s_g}}\left\lbrace \fint_{B_{2R}(x_0)}|Du|dx+G^{-1}\left[\fint_{B_{2R}(x_0)}[G(|D\psi|)+G(|\psi|)]dx \right] \right\rbrace ,
\end{equation*}
where $c=c(n,i_g,s_g,v,L)$.
\end{lemma}

\begin{proof}
Without loss of generality we may assume that $x_0=0$. We take $v=\frac{u+w}{2} \in u+W_{0}^{1,G}(B_{R})$ as comparison functions in the variational inequalities \eqref{vw1} and  \eqref{vw2}, which implies that
\begin{equation*}
\int_{B_{R}}[a(x,Du)-\overline{a}_{B_R(x_0)}(Dw)]\cdot D(u-w)dx\leqslant 0.
\end{equation*}
Then from \eqref{a(x)1}, \eqref{a(x)2},  Lemma \ref{adaog} and Lemma \ref{nibu},  we have
\begin{eqnarray*}
\fint_{B_{R}}G(|Du-Dw|)dx&\leqslant& \fint_{B_{R}}[\overline{a}_{B_R}(Du)-\overline{a}_{B_R}(Dw)]\cdot D(u-w)dx \\
&\leqslant& \fint_{B_{R}}[\overline{a}_{B_R}(Du)-a(x,Du)]\cdot D(u-w)dx \\
&\leqslant& \fint_{B_R}\theta(a,B_R)g(|Du|)|Du-Dw|dx \\
&\leqslant& c\varepsilon \fint_{B_R}G(|Du-Dw|)dx+c\fint_{B_R}\theta(a,B_R)G^{*}(g(|Du|))dx \\
&\leqslant& c\varepsilon \fint_{B_R}G(|Du-Dw|)dx+c\left( \fint_{B_R}\theta(a,B_R)^{\gamma'}dx\right)^{\frac{1}{\gamma'}}\left( \fint_{B_R}G(|Du|)^{\gamma}dx\right)^{\frac{1}{\gamma}} \\
&\leqslant& c\varepsilon \fint_{B_R}G(|Du-Dw|)dx+c\omega(R)\fint_{\frac{3R}{2}}G(|Du|)dx,
\end{eqnarray*}
where we used the fact $\theta \leqslant2L$.
Now we choose $\varepsilon$ small  enough to get
\begin{equation*}
\fint_{B_R}G(|Du-Dw|)dx\leqslant c \omega(R) \fint_{B_{\frac{3R}{2}}}G(|Du|)dx.
\end{equation*}
Because G is convex, we obtain
\begin{equation*}
G\left(\fint_{B_R}|Du-Dw|dx \right) \leqslant \fint_{B_R}G(|Du-Dw|)dx.
\end{equation*}
Using Remark \ref{remark1}, Lemma \ref{ag} and  Lemma \ref{reverse}, we conclude
\begin{eqnarray*}
\fint_{B_R}|Du-Dw|dx&\leqslant & G^{-1}\left(c \omega(R)\fint_{B_{\frac{3R}{2}}}G(|Du|) dx\right)  \\
&\leqslant & c \omega(R)^{\frac{1}{1+s_g}}G^{-1}\left( \fint_{B_{\frac{3R}{2}}}G(|Du|) dx\right) \\
&\leqslant & c \omega(R)^{\frac{1}{1+s_g}}\left\lbrace \fint_{B_{2R}}|Du|dx+G^{-1}\left[\fint_{B_{2R}}[G(|D\psi|)+G(|\psi|)]dx \right]\right\rbrace,
\end{eqnarray*}
and the proof is comlpete.

\end{proof}
The following two lemmas show some comparison estimates.  Since lemma 4.4 and lemma 4.5 in \cite{xiong1} give similar results, the two lemmas in our paper can be obtained by modifying their proving process a little.
\begin{lemma}\label{zabj}
Assume that conditions  \eqref{a(x)1}-\eqref{a(x)3} are fulfilled, let  $u \in W^{1,G}(B_R(x_0))$ with $u\geqslant \psi$ solves the inequality
\begin{equation*}
\int_{B_R(x_0)}a(x,Du)\cdot D(v-u)dx\geqslant 0
\end{equation*}
for any $v \in u+W_0^{1,G}(B_R(x_0))$ with $v\geqslant \psi \ a.e. \ on \ B_{R}(x_0)$. Let $w \in u+W_0^{1,G}(B_R(x_0))$  be a weak solution of the equation
\begin{equation*}
-\operatorname{div}\left(  a(x, Dw)\right)=  -\operatorname{div}\left(  a(x, D\psi)\right) \ \ \ \  on \ \ B_R(x_0)
\end{equation*}
and $ \psi \in W^{1,G}(B_R(x_0))\cap W^{2,1}(B_R(x_0)), \ \operatorname{div}\left(  a(x, D\psi)\right) \in L^1(B_R(x_0)).$
Then there exists $c=c(n,i_g,s_g,v,L)>0$ such that
\begin{equation}
\fint_{B_R(x_0)}|Du-Dw|dx\leqslant c\left(R\fint_{B_R(x_0)}(|\operatorname{div}\left(  a(x, D\psi)\right)|+1)dx\right)^{\frac{1}{i_g}}.
\end{equation}
\end{lemma}

\begin{lemma}\label{fbjg}
Assume that conditions  \eqref{a(x)1}-\eqref{a(x)3} are fulfilled, let  $f,g \in L^{1}(B_R(x_0))\cap (W^{1,G}(B_R(x_0)))'$ and
$u,w \in W^{1,G}(B_R(x_0))$ with $u-w \in W_{0}^{1,G}(B_R(x_0))$ be weak solutions of
\begin{equation}\label{fhg}
\left\{\begin{array}{r@{\ \ }c@{\ \ }ll}
-\operatorname{div}\left(  a(x, Du)\right) =f & \mbox{on}\ \ B_R(x_0)\,, \\[0.05cm]
-\operatorname{div}\left(  a(x, Dw)\right) =g & \mbox{on}\ \ B_R(x_0)\,. \\[0.05cm]
\end{array}\right.
\end{equation}
Then the following comparison estimates hold:
\begin{equation}
\fint_{B_R(x_0)}|Du-Dw|dx\leqslant c\left(R\fint_{B_R(x_0)}(|f|+|g|+1)dx \right)^{\frac{1}{i_g}},
\end{equation}
where $c=c(n,i_g,s_g,v,L)>0$.
\end{lemma}

Next we  state an excess decay estimate for a homogeneous comparison problem.
\begin{lemma}(see \cite{b13}, lemma 4.1)\label{lemma1.6}
If $ v\in W_{loc}^{1,G}(\Omega)$ is a local weak solution of
\begin{equation}
\operatorname{div}\left(  a(Dv)\right) =0 \ \ \ \ \ \ in\ \ \ \Omega
\end{equation}
under the assumption  \eqref{a(x)1} and \eqref{a(x)3}, then there exist constants $\beta \in (0,1)$ and $C=C(n,i_g,s_g,v,L)$ such that
\begin{equation}
\fint_{B_\rho(x_0)} \vert Dv-(Dv)_{B_\rho(x_0)}\vert \operatorname{d}\!\xi\leq C\left( \frac{\rho}{R}\right) ^\beta \fint_{B_R(x_0)} \vert Dv-(Dv)_{B_R(x_0)}\vert \operatorname{d}\!\xi,
\end{equation}
where $0<\rho\leqslant R, \ B_{2R}(x_0)\subset \Omega.$
\end{lemma}

We want to obtain a similar excess decay estimate, but with error terms, for solutions to  \eqref{u0}. And our approach is to transfer the excess decay estimate from Lemma \ref{lemma1.6} to solutions of an obstacle problem by employing a multistep comparison argument.
\begin{lemma}\label{zongjie}
Assume that conditions  \eqref{a(x)1}-\eqref{a(x)3}  are fulfilled, let  $B_{2R}(x_0) \subset \Omega,   \psi \in W^{1,G}(B_R(x_0))\cap W^{2,1}(B_R(x_0))$, $\frac{g(|D\psi|)}{|D\psi|}|D^{2}\psi| \in L^1(B_R(x_0)).$ Let $u\in W^{1,1}(B_R(x_0))$ with $u\geqslant \psi$ a.e. be a limit of approximating solutions for $OP(\psi; u)$ with measure data $\mu\in \mathcal{M}_{b}(B_R(x_0))$ (in the sense of Definition \ref{opdy}).
Then there exists $\beta \in (0,1)$ such that
\begin{eqnarray*}
&&\fint_{B_{\rho}(x_0)}|Du-(Du)_{B_{\rho}(x_0)}|dx\leqslant  c \left( \frac{\rho}{R}\right) ^{\beta}\fint_{B_R(x_0)}|Du-(Du)_{B_R(x_0)}|dx \\ &+& c\left(\frac{R}{\rho} \right) ^n \left[ \left[\frac{|\mu|(\overline{B_R(x_0)})}{R^{n-1}} \right] ^{\frac{1}{i_g}}+\left(R\fint_{B_R(x_0)}\left(\frac{g(|D\psi|)}{|D\psi|}|D^{2}\psi|+1\right)dx\right)^{\frac{1}{i_g}}\right] \\
&+&c\left(\frac{R}{\rho} \right) ^{n}  \omega(R)^{\frac{1}{1+s_g}}\left\lbrace \fint_{B_{R(x_0)}}|Du|dx+ G^{-1}\left[\fint_{B_{R(x_0)}}[G(|D\psi|)+G(|\psi|)]dx \right]\right\rbrace,
\end{eqnarray*}
where $0<\rho\leqslant R, \ c=c(n,i_g,s_g,v,L)$ and  $\beta$ is as in Lemma \ref{lemma1.6}.
\end{lemma}
\begin{remark}
Since we obtained all comparison estimates for regularized problems on $L^{1}-$level, our results also hold in measure data problems.
\end{remark}
\begin{proof}
Without loss of generality we may assume that $x_0=0$ and that $w_1, w_2, w_3 \in W^{1,G}(B_R)$ satisfy separately
\begin{equation*}
\left\{\begin{array}{r@{\ \ }c@{\ \ }ll}
&\int_{B_R}&  a(x, Dw_1)\cdot D(v-w_1)dx \geqslant0,\ \  \mbox{for} \ \ \forall \ v \in w_1+W_{0}^{1,G}(B_R) \ \mbox{with} \ v\geqslant \psi \  a.e. \ \mbox{on}\ \ B_R \,, \\[0.05cm]
&w_1&\geqslant \psi, w_1\geqslant0   \ \ \ \ \ \ a.e. \ \mbox{on}\ \ B_R\,, \\[0.05cm]
&w_1&=u \ \ \ \ \ \ \ \ \ \ \ \ \ \ \ \ \ \mbox{on}\ \ \partial B_R \,,
\end{array}\right.
\end{equation*}

\begin{equation*}
\left\{\begin{array}{r@{\ \ }c@{\ \ }ll}
&\int_{B_\frac{R}{2}}&  \overline{a}_{B_{\frac{R}{2}}}(Dw_2)\cdot D(v-w_2)dx \geqslant0,\ \  \mbox{for} \ \ \forall \ v \in w_1+W_{0}^{1,G}(B_{\frac{R}{2}}) \ \mbox{with} \ v\geqslant \psi \  a.e. \ \mbox{on}\ \ B_{\frac{R}{2}} \,, \\[0.05cm]
&w_2&\geqslant \psi \ \ \ \ \ \ \ \ \ \ \ \ \ \ \ \ \ a.e. \ \mbox{on}\ \ B_{\frac{R}{2}}\,, \\[0.05cm]
&w_2&=w_1 \ \ \ \ \ \ \ \ \ \ \ \ \ \ \ \ \ \mbox{on}\ \ \partial B_{\frac{R}{2}} \,,
\end{array}\right.
\end{equation*}

\begin{equation*}
\left\{\begin{array}{r@{\ \ }c@{\ \ }ll}
-\operatorname{div}\left(  \overline{a}_{B_{\frac{R}{2}}}(Dw_3)\right)&=&-\operatorname{div}\left( \overline{a}_{B_{\frac{R}{2}}}(D\psi)\right) \ \ \ \mbox{on}\ \ B_{\frac{R}{2}} \,, \\[0.05cm]
w_3&=&w_1  \ \  \ \ \ \ \ \ \ \ \ \ \ \ \ \ \ \ \ \ \ \ \ \  \ \ \mbox{on}\ \ \partial B_{\frac{R}{2}} \,, \\[0.05cm]
\end{array}\right.
\end{equation*}

\begin{equation*}
\left\{\begin{array}{r@{\ \ }c@{\ \ }ll}
-\operatorname{div}\left(  \overline{a}_{B_{\frac{R}{2}}}(Dw_4)\right)&=&0  \ \ \ \mbox{on}\ \ B_{\frac{R}{2}} \,, \\[0.05cm]
w_4&=&w_1  \ \ \ \mbox{on}\ \ \partial B_{\frac{R}{2}} \,. \\[0.05cm]
\end{array}\right.
\end{equation*}
In order to remove the inhomogeneity, we make use of  Corollary \ref{coro} to get  the comparison estimate
\begin{equation}\label{ll2}
\fint_{B_R} |Du-Dw_1| \operatorname{d}\!x \leqslant c\left[\frac{|\mu|(\overline{B_R})}{R^{n-1}} \right] ^{\frac{1}{i_g}}.
\end{equation}
Then we  use  Lemma \ref{diyi} to obtain
\begin{eqnarray}\label{ll5}
\fint_{B_{\frac{R}{2}}}|Dw_1-Dw_2|dx &\leqslant&  c \nonumber \omega(R)^{\frac{1}{1+s_g}}\left\lbrace \fint_{B_{R}}|Dw_1|dx+G^{-1}\left[\fint_{B_{R}}[G(|D\psi|)+G(|\psi|)]dx \right] \right\rbrace \\  \nonumber
&\leqslant& c \omega(R)^{\frac{1}{1+s_g}}\left\lbrace \fint_{B_{R}}|Du|dx+\left[\frac{|\mu|(\overline{B_R})}{R^{n-1}} \right] ^{\frac{1}{i_g}}\right\rbrace \\
&+&  c \omega(R)^{\frac{1}{1+s_g}}\left\lbrace G^{-1}\left[\fint_{B_{R}}[G(|D\psi|)+G(|\psi|)]dx \right] \right\rbrace.
\end{eqnarray}
Next, in order to transition to an obstacle-free problem, we  apply Lemma \ref{zabj} to get
\begin{equation}\label{l3}
\fint_{B_{\frac{R}{2}}}|Dw_2-Dw_3|dx\leqslant c\left(R\fint_{B_{\frac{R}{2}}}(|\operatorname{div}\left(  \overline{a}_{B_{\frac{R}{2}}}(D\psi)\right)|+1)dx\right)^{\frac{1}{i_g}}.
\end{equation}
Now we reduce to a homogeneous equation. An application of  Lemma \ref{fbjg} with $f=-\operatorname{div}\left(  \overline{a}_{B_{\frac{R}{2}}}(D\psi)\right)$ and $g=0$ implies that
\begin{equation}\label{l4}
\fint_{B_{\frac{R}{2}}}|Dw_3-Dw_4|dx\leqslant c\left(R\fint_{B_{\frac{R}{2}}}(|\operatorname{div}\left(  \overline{a}_{B_{\frac{R}{2}}}(D\psi)\right)|+1)dx\right)^{\frac{1}{i_g}}.
\end{equation}
Thanks to Lemma \ref{lemma1.6},  there exists $\beta \in (0,1)$ such that
\begin{eqnarray}\label{ll1}
 \nonumber  \fint_{B_\rho} \vert Dw_4-(Dw_4)_{B_\rho}\vert dx &\leqslant & c\left( \frac{\rho}{R}\right) ^\beta \fint_{B_{\frac{R}{2}}} \vert Dw_4-(Dw_4)_{B_{\frac{R}{2}}}\vert dx \\
&\leqslant & c\left( \frac{\rho}{R}\right) ^\beta \fint_{B_{\frac{R}{2}}}|Dw_4-Du| +\vert Du-(Du)_{B_R}\vert dx.
\end{eqnarray}
Finally we together with  \eqref{ll2} \eqref{ll5}   \eqref{l3} \eqref{l4} and  \eqref{ll1} ,  we infer
\begin{eqnarray*}
\fint_{B_{\rho}}|Du-(Du)_{B_{\rho}}|dx &\leqslant& \fint_{B_{\rho}}|Du-(Dw_4)_{B_{\rho}}|dx \\
&\leqslant& \fint_{B_{\rho}}(|Du-Dw_1|+|Dw_1-Dw_2|+|Dw_2-Dw_3| \\
&+&|Dw_3-Dw_4|+\vert Dw_4-(Dw_4)_{B_\rho}\vert )dx  \\
&\leqslant& c\left(\frac{R}{\rho} \right) ^n \left[ \left[\frac{|\mu|(\overline{B_{\frac{R}{2}})}}{R^{n-1}} \right] ^{\frac{1}{i_g}}+\left(R\fint_{B_{\frac{R}{2}}}(|\operatorname{div}\left(  \overline{a}_{B_{\frac{R}{2}}}(D\psi)\right)|+1)dx\right)^{\frac{1}{i_g}}\right] \\
&+&c\left(\frac{R}{\rho} \right) ^{n}  \omega(R)^{\frac{1}{1+s_g}}\left\lbrace \fint_{B_{R}}|Du|dx+\left[\frac{|\mu|(\overline{B_R})}{R^{n-1}} \right] ^{\frac{1}{i_g}} \right\rbrace \\
&+&c\left(\frac{R}{\rho} \right) ^{n}  \omega(R)^{\frac{1}{1+s_g}}G^{-1}\left[\fint_{B_{R}}[G(|D\psi|)+G(|\psi|)]dx \right] \\
&+& c\left( \frac{\rho}{R}\right) ^\beta \left[ \fint_{B_{\frac{R}{2}}}(|Dw_4-Du|+|Du-(Du)_{B_R}|)dx\right]  \\
&\leqslant& c \left( \frac{\rho}{R}\right) ^\beta  \fint_{B_R}|Du-(Du)_{B_R}|dx \\
&+& c\left(\frac{R}{\rho} \right) ^n \left[ \left[\frac{|\mu|(\overline{B_R})}{R^{n-1}} \right] ^{\frac{1}{i_g}}+\left(R\fint_{B_R}\left(\frac{g(|D\psi|)}{|D\psi|}|D^{2}\psi| +1\right)dx\right)^{\frac{1}{i_g}}\right] \\
&+&c\left(\frac{R}{\rho} \right) ^{n}  \omega(R)^{\frac{1}{1+s_g}}\left\lbrace \fint_{B_{R}}|Du|dx+ G^{-1}\left[\fint_{B_{R}}[G(|D\psi|)+G(|\psi|)]dx \right] \right\rbrace,
\end{eqnarray*}
where we used the fact $\omega \leqslant2L$ and the condition  \eqref{a(x)1} in the last step.
\end{proof}

\section{The proof of gradient estimates }\label{section5}
This section is devoted to obtain pointwise and oscillation estimates for the gradients of solutions  by  the sharp maximal function estimates.   We start with  a pointwise estimate of fractional maximal  operator by precise iteration methods.
\begin{proof}[Proof of Theorem \ref{Th1}]
\textup{\textbf{Proof of \eqref{1.11}}}
We  define
\begin{equation}
B_i:=B\left( x,\frac{R}{H^i}\right )=B(x,R_i),\ \ \mbox{for}\ \ \ i=0,1,2,...,\ \ \nonumber
\end{equation}
\begin{equation}
A_i:=\fint_{B_i}|Du-(Du)_{B_i}|dx,\ \ k_i:=|(Du)_{B_i}-S|,\ \ \ S\in \mathbb{R}^n.  \nonumber
\end{equation}
We take $H=H(n,i_g,s_g,v,L)>1$ large enough to have
\begin{equation}
c\left(\frac{1}{H} \right) ^\beta\leq \frac{1}{4}, \nonumber
\end{equation}
with $\beta$  as in lemma \ref{lemma1.6} and we apply lemma \ref{zongjie}  to obtain
\begin{eqnarray*}
\fint_{B_{i+1}}|Du-(Du)_{B_{i+1}}|\operatorname{d}\! \xi &\leqslant& \frac{1}{4}\fint_{B_{i}}|Du-(Du)_{B_{i}}|\operatorname{d}\!\xi+cH^{n} \left\lbrace  \left[\frac{|\mu|(\overline{B_i})}{R_{i}^{n-1}} \right] ^{\frac{1}{i_g}}+\left[\frac{D\Psi(B_i)}{R_{i}^{n-1}} \right] ^{\frac{1}{i_g}}\right\rbrace  \\
&+&cH^{n}  \omega(R_i)^{\frac{1}{1+s_g}}\left\lbrace \fint_{B_{i}}|Du|d\xi+ G^{-1}\left[\fint_{B_{i}}[G(|D\psi|)+G(|\psi|)]d\xi \right]\right\rbrace,
\end{eqnarray*}
Now we reduce the value of $R_0$-in a way depending only on $n,i_g,s_g,v,L$ and $\omega(\cdot)$- to get
\begin{equation*}
cH^{n}\omega(R_i)^{\frac{1}{1+s_g}}\leqslant cH^{n}\omega(R_0)^{\frac{1}{1+s_g}}\leqslant \frac{1}{4},
\end{equation*}
which together with  the following estimate
\begin{equation*}
\fint_{B_i}|Du|d\xi \leqslant \fint_{B_i}|Du-(Du)_{B_i}|d\xi+k_i+|S|,
\end{equation*}
we reduce that
\begin{eqnarray}\label{1.19}
\nonumber A_{i+1}&\leqslant& \frac{1}{2}A_i+c\left\lbrace \left[\frac{|\mu|(\overline{B_i})}{R_i^{n-1}} \right] ^{\frac{1}{i_g}} +\left[\frac{D\Psi(B_i)}{R_i^{n-1}} \right] ^{\frac{1}{i_g}}\right\rbrace  \\
&+&c[\omega(R_i)]^{\frac{1}{1+s_g}}\left\lbrace k_i+|S|+ G^{-1}\left[\fint_{B_{i}}[G(|D\psi|)+G(|\psi|)]d\xi \right]\right\rbrace
\end{eqnarray}
whenever $ i\geqslant0 $.  On the other hand, we calculate
\begin{eqnarray*}
|k_{i+1}-k_i|&\leq&|(Du)_{B_{i+1}}-(Du)_{B_i}| \\
&\leq& \fint_{B_{i+1}}|Du-(Du)_{B_{i}}|\operatorname{d}\!\xi\\
&\leq& H^n \fint_{B_i}|Du-(Du)_{B_{i}}|\operatorname{d}\!\xi=H^nA_i,
\end{eqnarray*}
from which we see that  for $ m\in \mathbb{N} $,
\begin{equation}\label{1.20}
k_{m+1}=\sum_{i=0}^m(k_{i+1}-k_i)+k_0\leq H^n\sum_{i=0}^mA_i+k_0.
\end{equation}
At this stage we sum up  \eqref{1.19}  over $ i \in \left\lbrace  0,...,m-1\right\rbrace$, which allow us to infer the inequality
\begin{eqnarray*}
\sum_{i=1}^mA_i&\leq& \frac{1}{2}\sum_{i=0}^{m-1}A_i+c\sum_{i=0}^{m-1}\left\lbrace \left[\frac{|\mu|(\overline{B_i})}{R_i^{n-1}} \right] ^{\frac{1}{i_g}} +\left[\frac{D\Psi(B_i)}{R_i^{n-1}} \right] ^{\frac{1}{i_g}}\right\rbrace  \\
&+&c\sum_{i=0}^{m-1}[\omega(R_i)]^{\frac{1}{1+s_g}}\left\lbrace k_i+|S|+ G^{-1}\left[\fint_{B_{i}}[G(|D\psi|)+G(|\psi|)]d\xi \right] \right\rbrace.
\end{eqnarray*}
Consequently,
\begin{eqnarray*}
\sum_{i=1}^mA_i&\leq& A_0+2c\sum_{i=0}^{m-1}\left\lbrace \left[\frac{|\mu|(\overline{B_i})}{R_i^{n-1}} \right] ^{\frac{1}{i_g}} +\left[\frac{D\Psi(B_i)}{R_i^{n-1}} \right] ^{\frac{1}{i_g}}\right\rbrace  \\
&+&c\sum_{i=0}^{m-1}[\omega(R_i)]^{\frac{1}{1+s_g}}\left\lbrace k_i+|S|+ G^{-1}\left[\fint_{B_{i}}[G(|D\psi|)+G(|\psi|)]d\xi \right] \right\rbrace.
\end{eqnarray*}
For every integer $m\geqslant1$ we employ \eqref{1.20} to gain
\begin{eqnarray}\label{2.3}
\nonumber k_{m+1}&\leq& cA_0+ck_0+c\sum_{i=0}^{m-1}\left\lbrace \left[\frac{|\mu|(\overline{B_i})}{R_i^{n-1}} \right] ^{\frac{1}{i_g}} +\left[\frac{D\Psi(B_i)}{R_i^{n-1}} \right] ^{\frac{1}{i_g}}\right\rbrace  \\
&+&c\sum_{i=0}^{m-1}[\omega(R_i)]^{\frac{1}{1+s_g}}\left\lbrace k_i+|S|+ G^{-1}\left[\fint_{B_{i}}[G(|D\psi|)+G(|\psi|)]d\xi \right] \right\rbrace.
\end{eqnarray}
we take into account the definition of $A_0$ to get
\begin{eqnarray}\label{a0guji}
\nonumber k_{m+1}&\leq& c\fint_{B_R}|Du-(Du)_{B_R}|+|Du-S|\operatorname{d}\!\xi+c\sum_{i=0}^{m-1}\left\lbrace \left[\frac{|\mu|(\overline{B_i})}{R_i^{n-1}} \right] ^{\frac{1}{i_g}} +\left[\frac{D\Psi(B_i)}{R_i^{n-1}} \right] ^{\frac{1}{i_g}}\right\rbrace  \\
&+&c\sum_{i=0}^{m-1}[\omega(R_i)]^{\frac{1}{1+s_g}}\left\lbrace k_i+|S|+ G^{-1}\left[\fint_{B_{i}}[G(|D\psi|)+G(|\psi|)]d\xi \right] \right\rbrace
\end{eqnarray}
for every $m\geqslant0.$
In the previous inequality we choose $S=0$ and multiply both sides by $R_{m+1}^{1-\alpha}$, taking into account that $\alpha\in[0,1]$ and $R_{m+1}\leq R_i$ for $0\leqslant i \leqslant m+1$, we get
\begin{eqnarray*}
R_{m+1}^{1-\alpha}k_{m+1}&\leq& cR^{1-\alpha} \fint_{B_R}|Du|\operatorname{d}\!\xi+c\sum_{i=0}^{m-1}R_i^{1-\alpha}\left\lbrace \left[\frac{|\mu|(\overline{B_i})}{R_i^{n-1}} \right] ^{\frac{1}{i_g}} +\left[\frac{D\Psi(B_i)}{R_i^{n-1}} \right] ^{\frac{1}{i_g}}\right\rbrace  \\
&+&c\sum_{i=0}^{m-1}R_i^{1-\alpha}[\omega(R_i)]^{\frac{1}{1+s_g}}\left\lbrace k_i+ G^{-1}\left[\fint_{B_{i}}[G(|D\psi|)+G(|\psi|)]d\xi \right] \right\rbrace
\end{eqnarray*}
and therefore
\begin{eqnarray}\label{1.22}
\nonumber R_{m+1}^{1-\alpha}k_{m+1}&\leq& cR^{1-\alpha} \fint_{B_R}|Du|\operatorname{d}\!\xi+c\sum_{i=0}^{m}\left\lbrace \left[\frac{|\mu|(\overline{B_i})}{R_i^{n-1-i_g(1-\alpha)}} \right] ^{\frac{1}{i_g}}+\left[\frac{D\Psi(B_i)}{R_i^{n-1-i_g(1-\alpha)}} \right] ^{\frac{1}{i_g}}\right\rbrace \\
&+&c\sum_{i=0}^{m}R_i^{1-\alpha}[\omega(R_i)]^{\frac{1}{1+s_g}}\left\lbrace k_i+ G^{-1}\left[\fint_{B_{i}}[G(|D\psi|)+G(|\psi|)]d\xi \right] \right\rbrace.
\end{eqnarray}
Now we employ the definition of Wolff potential to estimate the last term on the right-hand side in \eqref{1.22} and we find
\begin{eqnarray}\label{1.23}
\nonumber&&\sum_{i=0}^{\infty}\left[\frac{|\mu|(\overline{B_i})}{R_i^{n-1-i_g(1-\alpha)}} \right] ^{\frac{1}{i_g}} \\
&\leq& \frac{c}{\log2}\int_R^{2R}\left[\frac{|\mu|(\overline{B_\rho})}{\rho^{n-1-i_g(1-\alpha)}} \right]  ^{\frac{1}{i_g}}\frac{\operatorname{d}\!\rho}{\rho}\nonumber
+\sum_{i=0}^{\infty}\frac{c}{\log H}\int_{R_{i+1}}^{R_i}\left[\frac{|\mu|(\overline{B_\rho})}{\rho^{n-1-i_g(1-\alpha)}} \right]  ^{\frac{1}{i_g}}\frac{\operatorname{d}\!\rho}{\rho}\\
&\leq& cW_{1-\alpha+\frac{\alpha}{i_g+1},i_g+1}^{\mu}(x,2R),
\end{eqnarray}
in the last step, we apply $|\mu|(B_{2R})<+\infty$ by assumption, then $|\mu|(\partial B_{\rho})>0$ can hold at most for countably many radii $\rho \in (R,2R)$.
Similarly, we have
\begin{equation*}
\sum_{i=0}^{\infty}\left[\frac{D\Psi(B_i)}{R_i^{n-1-i_g(1-\alpha)}} \right] ^{\frac{1}{i_g}} \leqslant cW_{1-\alpha+\frac{\alpha}{i_g+1},i_g+1}^{[\psi]}(x,2R),
\end{equation*}
\begin{equation*}
\sum_{i=0}^{\infty}\left[\omega(R_i)\right] ^{\frac{1}{1+s_g}}\leqslant c\int_{0}^{2R}[\omega(\rho)]^{\frac{1}{1+s_g}}\frac{d\rho}{\rho},
\end{equation*}
\begin{eqnarray*}
\sum_{i=0}^{\infty}R_{i}^{1-\alpha}\left[\omega(R_i)\right] ^{\frac{1}{1+s_g}}G^{-1}\left[\fint_{B_{i}}[G(|D\psi|)+G(|\psi|)]d\xi \right] \\
\leqslant c\int_{0}^{2R}[\omega(\rho)]^{\frac{1}{1+s_g}}G^{-1}\left[\fint_{B_{\rho}}[G(|D\psi|)+G(|\psi|)]d\xi \right]\frac{d\rho}{\rho^{\alpha}}.
\end{eqnarray*}
Consequently,
\begin{equation}\label{2.7}
R_{m+1}^{1-\alpha}k_{m+1}\leqslant cM+c\sum_{i=0}^{m}R_{i}^{1-\alpha}k_{i}\left[\omega(R_i)\right] ^{\frac{1}{1+s_g}},
\end{equation}
where
\begin{eqnarray*}
M:&=&c\left[ R^{1-\alpha}\fint_{B_R}|Du|d\xi +W_{1-\alpha+\frac{\alpha}{i_g+1},i_g+1}^{\mu}(x,2R)+W_{1-\alpha+\frac{\alpha}{i_g+1},i_g+1}^{[\psi]}(x,2R)\right]  \\
&+&c\int_{0}^{2R}[\omega(\rho)]^{\frac{1}{1+s_g}}G^{-1}\left[\fint_{B_{\rho}}[G(|D\psi|)+G(|\psi|)]d\xi \right]\frac{d\rho}{\rho^{\alpha}}.
\end{eqnarray*}
We now prove by induction that
\begin{equation}\label{mguji}
R_{m+1}^{1-\alpha}k_{m+1}\leqslant (c+c^{*})M,
\end{equation}
holds for every $m\geqslant 0$, some positive constants $c,c^{*}>1$.

Firstly, the case $m=0$ of \eqref{2.7} is trivial. Then we assume that $R_{i}^{1-\alpha}k_{i}\leqslant (c+c^{*})M$ for $i\leqslant m$ and prove it for $m+1$. Because of \eqref{dytj}, we further reduce the value of $R_0$ to get
\begin{equation*}
\int_{0}^{2R}[\omega(\rho)]^{\frac{1}{1+s_g}}\frac{d \rho}{\rho}\leqslant \int_{0}^{2R_0}[\omega(\rho)]^{\frac{1}{1+s_g}}\frac{d \rho}{\rho}\leqslant \frac{1}{2(c+c^{*})c}.
\end{equation*}
Then taking \eqref{2.7} into account we obtain
\begin{equation*}
R_{m+1}^{1-\alpha}k_{m+1}\leqslant cM+c\sum_{i=0}^{m}(c+c^{*})M\left[\omega(R_i)\right] ^{\frac{1}{1+s_g}}\leqslant (c+c^{*})M,
\end{equation*}
therefore \eqref{mguji} follows for every integer $m\geqslant0$.
Now  we define
\begin{equation*}
C_m:=R_m^{1-\alpha}A_m=R_m^{1-\alpha}\fint_{B_m}|Du-(Du)_{B_m}|d\xi
\end{equation*}
\begin{equation*}
h_m:=\fint_{B_m}|Du|d\xi.
\end{equation*}
Then it's  easy to obtain
\begin{eqnarray*}
R_m^{1-\alpha}h_m&=&R_m^{1-\alpha}\fint_{B_m}|Du|d\xi \\
&\leq& R_m^{1-\alpha}\fint_{B_m}|Du-(Du)_{B_m}|+|(Du)_{B_m}|d\xi  \\
&=&R_m^{1-\alpha}k_m+C_m \\
&\leq& CM+C_m
\end{eqnarray*}
with $M$ as in \eqref{2.7}, and so we just need to look for  a bound on $C_m$.
Applying \eqref{1.23}  and keeping in mind the definition of $M$ in\eqref{2.7}, we  gain
\begin{eqnarray*}
\left[\frac{|\mu|(\overline{B_i})}{R_i^{n-1}} \right] ^{\frac{1}{i_g}}
&\leq& CR_i^{\alpha-1}W_{1-\alpha+\frac{\alpha}{i_g+1},i_g+1}^{\mu}(x,2R) \\
&\leq&  CR_i^{\alpha-1}M.
\end{eqnarray*}
Similarly, we have
\begin{equation*}
\left[\frac{D\Psi(B_i)}{R_{i}^{n-1}} \right] ^{\frac{1}{i_g}}\leqslant cR_{i}^{\alpha-1}M,
\end{equation*}
\begin{equation*}
\left[\omega(R_i) \right] ^{\frac{1}{1+s_g}} G^{-1}\left[\fint_{B_{i}}[G(|D\psi|)+G(|\psi|)]d\xi \right] \leqslant cR_{i}^{\alpha-1}M.
\end{equation*}
Therefore, referring to  \eqref{1.19}  we have
\begin{equation*}
A_{m+1}\leq \frac{1}{2}A_m+cR_m^{\alpha-1}M +ck_m.
\end{equation*}
In turn, we make use of  $\eqref{mguji}$ to obtain
\begin{equation*}
A_{m+1}\leq \frac{1}{2}A_m+cR_m^{\alpha-1}M.
\end{equation*}
then multiply both sides by $R_{m+1}^{1-\alpha}$, we get
\begin{eqnarray}\label{1.25}
C_{m+1}&\leq& \frac{1}{2}\left( \frac{R_{m+1}}{R_m}\right) ^{1-\alpha}C_m \nonumber
+C\left( \frac{R_{m+1}}{R_m}\right) ^{1-\alpha}M \\
&\leq& \frac{1}{2}C_m+C_1M.
\end{eqnarray}
Now we shall prove by induction that
\begin{equation}\label{2.9}
C_m\leq 2C_1M
\end{equation}
holds whenever $m\geqslant0$. When $k=0$, we have
\begin{eqnarray*}
C_0&=&R^{1-\alpha}\fint_{B_R}|Du-(Du)_{B_R}|\operatorname{d}\!\xi \\
&\leq& 2R^{1-\alpha}\fint_{B_R}|Du|\operatorname{d}\!\xi\\
&\leq& 2M.
\end{eqnarray*}
Then we assume  \eqref{2.9} holds for  $k=m-1$  , then using \eqref{1.25}  we gain
\begin{equation*}
C_m\leq 2C_1M\ \ \ \ \ \mbox{for}\ \ \ m\geqslant0.
\end{equation*}
Now we consider $r\leqslant R$ and determine the integer $i\geqslant0$ such that $R_{i+1}\leqslant r\leqslant R_i$, then we have
\begin{eqnarray*}
r^{1-\alpha}\fint_{B_r}|Du|\operatorname{d}\!\xi
&\leq& \left(\frac{R_i}{R_{i+1}} \right) ^nR_i^{1-\alpha}\fint_{B_i}|Du|\operatorname{d}\!\xi \\
&\leq& CH^nR_i^{1-\alpha}h_i \\
&\leq& CM.
\end{eqnarray*}
Recalling the definition of $M$ and  the restricted maximal operator   we in turn obtain
\begin{eqnarray*}
M_{1-\alpha,R}(|Du|)(x)&\leqslant &c\left[ R^{1-\alpha}\fint_{B_R}|Du|d\xi+W_{1-\alpha+\frac{\alpha}{i_g+1},i_g+1}^{\mu}(x,2R)+W_{1-\alpha+\frac{\alpha}{i_g+1},i_g+1}^{[\psi]}(x,2R)\right]  \\
&+&c\int_{0}^{2R}[\omega(\rho)]^{\frac{1}{1+s_g}}G^{-1}\left[\fint_{B_{\rho}}[G(|D\psi|)+G(|\psi|)]d\xi \right]\frac{d\rho}{\rho^{\alpha}}.
\end{eqnarray*}
Moreover, because of
\begin{equation*}
M_{\alpha,R}^{\#}(u)(x)\leqslant M_{1-\alpha,R}(|Du|)(x),
\end{equation*}
then we have
\begin{eqnarray*}
&&M_{\alpha,R}^{\#}(u)(x)+M_{1-\alpha,R}(Du)(x)\\
&\leqslant& c\left[ R^{1-\alpha}\fint_{B_R}|Du|d\xi+W_{1-\alpha+\frac{\alpha}{i_g+1},i_g+1}^{\mu}(x,2R)+W_{1-\alpha+\frac{\alpha}{i_g+1},i_g+1}^{[\psi]}(x,2R)\right]  \\
&+&c\int_{0}^{2R}[\omega(\rho)]^{\frac{1}{1+s_g}}G^{-1}\left[\fint_{B_{\rho}}[G(|D\psi|)+G(|\psi|)]d\xi \right]\frac{d\rho}{\rho^{\alpha}}.
\end{eqnarray*}
\textup{\textbf{Proof of \eqref{1.122}}}
We define
\begin{equation}\label{a}
\widehat{A}_i:=R_i^{-\alpha}\fint_{B_i}|Du-(Du)_{B_i}|\operatorname{d}\!\xi
\end{equation}
Using lemma \ref{zongjie} (multiply both sides by $R_{i+1}^{-\alpha}$) we obtain
\begin{eqnarray*}
\widehat{A}_{i+1}
&\leqslant &c\left( \frac{R_{i+1}}{R_{i}}\right) ^{\beta-\alpha}\widehat{A}_i+c \left( \frac{R_{i}}{R_{i+1}}\right) ^{n+\alpha}\left\lbrace \left[ \frac{|\mu|(\overline{B_i})}{R_i^{n-1+\alpha i_g}}\right] ^{\frac{1}{i_g}}+ \left[ \frac{D\Psi(B_i)}{R_i^{n-1+\alpha i_g}}\right] ^{\frac{1}{i_g}}\right\rbrace \\
&+&c\left( \frac{R_{i}}{R_{i+1}}\right) ^{n+\alpha}\frac{1}{R_i^{\alpha}} \omega(R_i)^{\frac{1}{1+s_g}}\left\lbrace \fint_{B_{i}}|Du|d\xi+ G^{-1}\left[\fint_{B_{i}}[G(|D\psi|)+G(|\psi|)]d\xi \right] \right\rbrace.
\end{eqnarray*}
We choose $H=H(n,i_a,s_a,\widehat{\alpha},\beta)>1$ large enough to gain
\begin{equation*}
c\left( \frac{R_{i+1}}{R_{i}}\right) ^{\beta-\alpha}
=c\left( \frac{1}{H}\right)^{\beta-\alpha}
\leqslant c\left( \frac{1}{H}\right)^{\beta-\widehat{\alpha} }
\leqslant\frac{1}{2}
\end{equation*}
and
\begin{equation*}
 \frac{|\mu|(\overline{B_i})}{R_i^{n-1+\alpha i_g}}\leqslant H^{n-1+\alpha i_g}\frac{|\mu|(B_{i-1})}{R_{i-1}^{n-1+\alpha i_g}}.
\end{equation*}
From assumption \eqref{wtiaojian}, we have
$$\frac{[\omega(R_i)]^{\frac{1}{1+s_g}}}{R_{i}^{\alpha}}\leqslant \frac{[\omega(R_i)]^{\frac{1}{1+s_g}}}{R_{i}^{\widehat{\alpha}}}\leqslant c_0,$$
then because of  the definition of restricted maximal operator, we conclude that
\begin{eqnarray*}
\widehat{A}_{i+1}&\leqslant& \frac{1}{2}\widehat{A}_{i}+c\left\lbrace \left[ M_{1-\alpha i_g,R}(\mu)(x)\right] ^{\frac{1}{i_g}}+\left[ \overline{M}_{1-\alpha i_g,R}(\psi)(x)\right] ^{\frac{1}{i_g}}+\fint_{B_{i}}|Du|d\xi \right\rbrace \\
&+&c\frac{1}{R_i^{\alpha}} \omega(R_i)^{\frac{1}{1+s_g}}\left\lbrace  G^{-1}\left[\fint_{B_{i}}[G(|D\psi|)+G(|\psi|)]d\xi\right] \right\rbrace
\end{eqnarray*}
for every $i\geqslant0$.
Moreover, from the case $\alpha=1$ of inequality \eqref{1.11}, we have
\begin{eqnarray}\label{mwanguji}
\nonumber \fint_{B_i}|Du|d\xi &\leqslant& c\left[ \fint_{B_R}|Du|d\xi+W_{\frac{1}{i_g+1},i_g+1}^{\mu}(x,2R)+W_{ \frac{1}{i_g+1},i_g+1}^{[\psi]}(x,2R)\right]  \\ \nonumber
&+&c\int_{0}^{2R}[\omega(\rho)]^{\frac{1}{1+s_g}}G^{-1}\left[\fint_{B_{\rho}}[G(|D\psi|)+G(|\psi|)]d\xi \right]\frac{d\rho}{\rho} \\
&:=&cM^{*}.
\end{eqnarray}
So combined with the two previous inequalities, we obtain
\begin{eqnarray*}
\widehat{A}_{i+1}&\leqslant& \frac{1}{2}\widehat{A}_{i}+c\left\lbrace \left[ M_{1-\alpha i_g,R}(\mu)(x)\right] ^{\frac{1}{i_g}}+\left[ \overline{M}_{1-\alpha i_g,R}(\psi)(x)\right] ^{\frac{1}{i_g}}+M^{*}\right\rbrace \\
&+&c\frac{1}{R_i^{\alpha}} \omega(R_i)^{\frac{1}{1+s_g}}\left\lbrace  G^{-1}\left[\fint_{B_{i}}[G(|D\psi|)+G(|\psi|)]d\xi \right] \right\rbrace.
\end{eqnarray*}
Iterating the previous relation, we conclude
\begin{eqnarray*}
 \widehat{A}_{i}&\leqslant& 2^{-i} \widehat{A}_{0}+c\sum_{j=0}^{i-1}2^{-j} \left\lbrace \left[ M_{1-\alpha i_g,R}(\mu)(x)\right] ^{\frac{1}{i_g}}+\left[ \overline{M}_{1-\alpha i_g,R}(\psi)(x)\right] ^{\frac{1}{i_g}}+M^{*}\right\rbrace  \\
 &+& c\sum_{j=0}^{i-1}2^{-j}  \frac{1}{R_j^{\alpha}} \omega(R_j)^{\frac{1}{1+s_g}}\left\lbrace   G^{-1}\left[\fint_{B_{j}}[G(|D\psi|)+G(|\psi|)]d\xi \right] \right\rbrace
\end{eqnarray*}
holds for every $i\geqslant1$.
Then similar to \eqref{1.23}, we reduce
\begin{eqnarray*}
&&\sum_{j=0}^{i-1}2^{-j}  \frac{1}{R_j^{\alpha}} \omega(R_j)^{\frac{1}{1+s_g}}\left\lbrace   G^{-1}\left[\fint_{B_{j}}[G(|D\psi|)+G(|\psi|)]d\xi \right] \right\rbrace  \\
&\leqslant &c\int_{0}^{2R}[\omega(\rho)]^{\frac{1}{1+s_g}}G^{-1}\left[\fint_{B_{\rho}}[G(|D\psi|)+G(|\psi|)]d\xi \right]\frac{d\rho}{\rho^{1+\alpha}}.
\end{eqnarray*}
Then recalling  \eqref{a}  and  \eqref{1.8} one easily deduces that
\begin{eqnarray*}
 \sup_{i\geqslant0}\widehat{A}_{i}&\leqslant& c\left\lbrace R^{-\alpha}\fint_{B_R}|Du|\operatorname{d}\!\xi+ \left[  M_{1-\alpha i_g,R}(\mu)(x)\right] ^{\frac{1}{i_g}}+\left[  \overline{M}_{1-\alpha i_g,R}(\psi)(x)\right] ^{\frac{1}{i_g}}\right\rbrace  \\
 &+& c\left\lbrace M^{*}+\int_{0}^{2R}[\omega(\rho)]^{\frac{1}{1+s_g}}G^{-1}\left[\fint_{B_{\rho}}[G(|D\psi|)+G(|\psi|)]d\xi \right]\frac{d\rho}{\rho^{1+\alpha}}
\right\rbrace .
\end{eqnarray*}
For every $ \rho\in (0,R]$, let $ i\in \mathbb{N}$ be  such that $ R_{i+1}<\rho\leqslant R_{i} $,  then we gain
\begin{eqnarray*}
\rho^{-\alpha}\fint_{B_{\rho}}\vert Du-(Du)_{B_{\rho}}\vert\operatorname{d}\!\xi &\leqslant & c\frac{R_{i}^n}{\rho^n}R_{i+1}^{-\alpha}\fint_{B_i}\vert Du-(Du)_{B_i}\vert\operatorname{d}\!\xi  \\
&\leqslant & c\sup_{i\geqslant0}\widehat{A}_{i},
\end{eqnarray*}
from which we obtain that
\begin{eqnarray*}
M_{\alpha,R}^{\#}(Du)(x)&\leqslant& c\left\lbrace R^{-\alpha}\fint_{B_R}|Du|\operatorname{d}\!\xi+ \left[  M_{1-\alpha i_g,R}(\mu)(x)\right] ^{\frac{1}{i_g}}+\left[  \overline{M}_{1-\alpha i_g,R}(\psi)(x)\right] ^{\frac{1}{i_g}}\right\rbrace  \\
&+&c\left\lbrace W_{\frac{1}{i_g+1},i_g+1}^{\mu}(x,2R)+W_{ \frac{1}{i_g+1},i_g+1}^{[\psi]}(x,2R)\right\rbrace  \\
&+&c\int_{0}^{2R}[\omega(\rho)]^{\frac{1}{1+s_g}}G^{-1}\left[\fint_{B_{\rho}}[G(|D\psi|)+G(|\psi|)]d\xi \right]\frac{d\rho}{\rho^{1+\alpha}}.
\end{eqnarray*}
 This completes the proof of Theorem \ref{Th1}.
\end{proof}
\begin{proof}[Proof of Theorem \ref{Th2}]
At first we give the proof of the estimate \eqref{du}.

We choose $S=0$ and make use  of \eqref{a0guji}, we conclude
\begin{eqnarray*}
\nonumber k_{m+1}&\leq& c\fint_{B_R(x_0)}|Du-(Du)_{B_R(x_0)}|+|Du|\operatorname{d}\! x+c\sum_{i=0}^{m-1}\left\lbrace \left[\frac{|\mu|(\overline{B_{R_i}}(x_0))}{R_i^{n-1}} \right] ^{\frac{1}{i_g}} +\left[\frac{D\Psi(B_{R_i}(x_0))}{R_i^{n-1}} \right] ^{\frac{1}{i_g}}\right\rbrace  \\
&+&c\sum_{i=0}^{m-1}[\omega(B_i)]^{\frac{1}{1+s_g}}\left\lbrace \fint_{B_{R_i}(x_0)}|Du|dx+ G^{-1}\left[\fint_{B_{R_i}(x_0)}[G(|D\psi|)+G(|\psi|)]dx \right]\right\rbrace.
\end{eqnarray*}
On the other hand, we observe
\begin{equation*}
\sum_{i=0}^{+\infty}\left[\frac{|\mu|(\overline{B_{R_i}}(x_0))}{R_i^{n-1}} \right] ^{\frac{1}{i_g}} \leqslant cW_{\frac{1}{i_g+1},i_g+1}^{\mu}(x_0,2R),
\end{equation*}
\begin{equation*}
\sum_{i=0}^{+\infty}\left[\frac{D\Psi(B_{R_i}(x_0))}{R_i^{n-1}} \right] ^{\frac{1}{i_g}} \leqslant cW_{\frac{1}{i_g+1},i_g+1}^{[\psi]}(x_0,2R),
\end{equation*}
\begin{equation*}
\sum_{i=0}^{+\infty}\left[\omega(R_i)\right] ^{\frac{1}{1+s_g}}\leqslant c\int_{0}^{2R}[\omega(\rho)]^{\frac{1}{1+s_g}}\frac{d\rho}{\rho}\leqslant c,
\end{equation*}
\begin{eqnarray*}
\sum_{i=0}^{+\infty}\left[\omega(R_i)\right] ^{\frac{1}{1+s_g}}G^{-1}\left[\fint_{B_{R_i}(x_0)}[G(|D\psi|)+G(|\psi|)]dx \right] \\
\leqslant c\int_{0}^{2R}[\omega(\rho)]^{\frac{1}{1+s_g}}G^{-1}\left[\fint_{B_{\rho}(x_0)} [G(|D\psi|)+G(|\psi|)]dx\right]\frac{d\rho}{\rho},
\end{eqnarray*}
Coupling \eqref{mwanguji} with above estimates, it follows that
\begin{eqnarray*}
|Du(x_0)|=\lim_{m\rightarrow\infty}k_{m+1} &\leqslant& c\left[ \fint_{B_R(x_0)}|Du|dx+W_{\frac{1}{i_g+1},i_g+1}^{\mu}(x_0,2R)+W_{ \frac{1}{i_g+1},i_g+1}^{[\psi]}(x_0,2R)\right] \\
&+&c\int_{0}^{2R}[\omega(\rho)]^{\frac{1}{1+s_g}}G^{-1}\left[\fint_{B_{\rho}(x_0)}[G(|D\psi|)+G(|\psi|)]dx \right]\frac{d\rho}{\rho}.
\end{eqnarray*}

Next we prove the estimate \eqref{du-du}.
For every $ x,y\in B_\frac{R}{4}(x_0)$,  we define
\begin{equation*}
r_i:=\frac{r}{H^i}, \ \ \ \ r\leqslant \frac{R}{2}, \ \ \ \ k_i=|(Du)_{B_{r_i}(x)}-S|, \ \ \ \ \overline{k_i}=|(Du)_{B_{r_i}(y)}-S|.
\end{equation*}
Taking advantage  of   \eqref{a0guji} again, we obtain
\begin{eqnarray*}
k_{m+1}&\leqslant & c\fint_{B_r(x)}(\vert Du-(Du)_{B_r(x)}\vert+\vert Du-S\vert)\operatorname{d}\!\xi \\
&+&c r^{\alpha}\sum_{i=0}^{m-1}\left\lbrace \left[ \frac{|\mu|(\overline{B_{r_i}}(x))}{r_i^{n-1+\alpha i_g}}\right] ^\frac{1}{i_g}+\left[ \frac{D\Psi(B_{r_i}(x))}{r_i^{n-1+\alpha i_g}}\right] ^\frac{1}{i_g}\right\rbrace \\
&+&cr^{\alpha}\sum_{i=0}^{m-1}\frac{1}{r_{i}^{\alpha}}[\omega(r_i)]^{\frac{1}{1+s_g}}\left\lbrace \fint_{B_{r_i}(x)}|Du|d\xi+ G^{-1}\left[\fint_{B_{r_i}(x)}[G(|D\psi|)+G(|\psi|)]d\xi \right] \right\rbrace.
\end{eqnarray*}
Moreover,
\begin{equation*}
\sum_{i=0}^{+\infty}\left[\frac{|\mu|(\overline{B_{r_i}(x)})}{r_i^{n-1+\alpha i_g}} \right] ^{\frac{1}{i_g}}
\leqslant cW_{\frac{1-\alpha i_{g}}{1+i_g},i_g+1}^{\mu}(x,2r),
\end{equation*}
\begin{equation*}
\sum_{i=0}^{+\infty}\left[\frac{D\Psi(B_{r_i}(x))}{r_i^{n-1+\alpha i_g}} \right] ^{\frac{1}{i_g}} \leqslant cW_{\frac{1-\alpha i_g}{1+i_g},i_g+1}^{[\psi]}(x,2r),
\end{equation*}
\begin{equation*}
\sum_{i=0}^{+\infty}\frac{1}{r_{i}^{\alpha}}\left[\omega(r_i)\right] ^{\frac{1}{1+s_g}}\leqslant c\int_{0}^{2r}[\omega(\rho)]^{\frac{1}{1+s_g}}\frac{d\rho}{\rho^{1+\alpha}}\leqslant c,
\end{equation*}
\begin{eqnarray*}
\sum_{i=0}^{+\infty}\frac{1}{r_{i}^{\alpha}}\left[\omega(R_i)\right] ^{\frac{1}{1+s_g}}G^{-1}\left[\fint_{B_{r_i}(x_0)}[G(|D\psi|)+G(|\psi|)]d\xi \right] \\
\leqslant c\int_{0}^{2r}[\omega(\rho)]^{\frac{1}{1+s_g}}G^{-1}\left[\fint_{B_{\rho}(x)}[G(|D\psi|)+G(|\psi|)]d\xi \right]\frac{d\rho}{\rho^{1+\alpha}}.
\end{eqnarray*}
Combining \eqref{mwanguji} with the previous estimates, so we have
\begin{eqnarray*}
k_{m+1} &\leqslant& c \fint_{B_r(x)}(\vert Du-(Du)_{B_r(x)}\vert+\vert Du-S\vert)\operatorname{d}\!\xi \\
&+&c r^{\alpha}\left[  W^{\mu}_{-\alpha+\frac{1+\alpha}{1+i_g},i_g+1}(x,R)+ W^{[\psi]}_{-\alpha+\frac{1+\alpha}{1+i_g},i_g+1}(x,R)\right] \\
&+&c r^{\alpha} \left\lbrace \fint_{B_r(x)}|Du|d\xi+\int_{0}^{R}[\omega(\rho)]^{\frac{1}{1+s_g}}G^{-1}\left[\fint_{B_{\rho}(x)}[G(|D\psi|)+G(|\psi|)]d\xi \right]\frac{d\rho}{\rho^{1+\alpha}} \right\rbrace.
\end{eqnarray*}
If $x$ is the Lebesgue's point of $Du$, then let $m\rightarrow\infty$, we derive
\begin{eqnarray*}
|Du(x)-S|&=&\lim_{m\rightarrow\infty}k_{m+1}  \\
&\leqslant &c \fint_{B_r(x)}(\vert Du-(Du)_{B_r(x)}\vert+\vert Du-S\vert)\operatorname{d}\!\xi \\
&+&c r^{\alpha}\left[  W^{\mu}_{-\alpha+\frac{1+\alpha}{1+i_g},i_g+1}(x,R)+ W^{[\psi]}_{-\alpha+\frac{1+\alpha}{1+i_g},i_g+1}(x,R)\right] \\
&+&c r^{\alpha} \left\lbrace \fint_{B_r(x)}|Du|d\xi+\int_{0}^{R}[\omega(\rho)]^{\frac{1}{1+s_g}}G^{-1}\left[\fint_{B_{\rho}(x)}[G(|D\psi|)+G(|\psi|)]d\xi \right]\frac{d\rho}{\rho^{1+\alpha}} \right\rbrace.
\end{eqnarray*}
If $y$ is the Lebesgue's point of $Du$, we have a similar result.
 Then coupling with the previous two estimates tells us
\begin{eqnarray*}
&& |Du(x)-Du(y)|\\
 &\leqslant &c \fint_{B_r(x)}(\vert Du-(Du)_{B_r(x)}\vert+\vert Du-S\vert)\operatorname{d}\!\xi \\
&+&c r^{\alpha}\left[  W^{\mu}_{-\alpha+\frac{1+\alpha}{1+i_g},i_g+1}(x,R)+ W^{[\psi]}_{-\alpha+\frac{1+\alpha}{1+i_g},i_g+1}(x,R)\right] \\
&+&c r^{\alpha} \left\lbrace \fint_{B_r(x)}|Du|d\xi+\int_{0}^{R}[\omega(\rho)]^{\frac{1}{1+s_g}}G^{-1}\left[\fint_{B_{\rho}(x)}[G(|D\psi|)+G(|\psi|)]d\xi \right]\frac{d\rho}{\rho^{1+\alpha}} \right\rbrace \\
&+&c \fint_{B_r(y)}(\vert Du-(Du)_{B_r(y)}\vert+\vert Du-S\vert)\operatorname{d}\!\xi \\
&+&c r^{\alpha}\left[  W^{\mu}_{-\alpha+\frac{1+\alpha}{1+i_g},i_g+1}(y,R)+ W^{[\psi]}_{-\alpha+\frac{1+\alpha}{1+i_g},i_g+1}(y,R)\right] \\
&+&c r^{\alpha} \left\lbrace \fint_{B_r(y)}|Du|d\xi+\int_{0}^{R}[\omega(\rho)]^{\frac{1}{1+s_g}}G^{-1}\left[\fint_{B_{\rho}(y)}[G(|D\psi|)+G(|\psi|)]d\xi \right]\frac{d\rho}{\rho^{1+\alpha}} \right\rbrace.
\end{eqnarray*}
We now choose
\begin{equation*}
S:=(Du)_{B_{3r}(x)}, \ \ \ r:=\dfrac{|x-y|}{2},
\end{equation*}
it's easy to see that  $ B_{r}(y)\subseteq B_{3r}(x) $ and therefore
\begin{eqnarray*}
&&\fint_{B_r(x)}(\vert Du-(Du)_{B_r(x)}\vert+\vert Du-S\vert)\operatorname{d}\!\xi+
\fint_{B_{r}(y)}(\vert Du-(Du)_{B_{r}(y)}\vert+\vert Du-S\vert)\operatorname{d}\!\xi  \\
&\leqslant &c(n)\fint_{B_{3r}(x)}\vert Du-(Du)_{B_{3r}(x)}\vert\operatorname{d}\!\xi.
\end{eqnarray*}
Now notice that $ x,y\in B_\frac{R}{4}(x_0)$, so $ |x-y|\leqslant \frac{R}{2} $ and then $B_{3r}(x)\subseteq B_{\frac{3R}{4}}(x)\subseteq B_R(x_0)$.
Therefore apply \eqref{1.122} to obtain
\begin{eqnarray*}
\fint_{B_{3r}(x)}\vert Du-(Du)_{B_{3r}(x)}\vert\operatorname{d}\!\xi
&\leqslant&  c r^{\alpha}M^\#_{\alpha,\frac{3R}{4}}(Du)(x)  \\
&\leqslant&c \left( \frac{r}{R}\right) ^{\alpha}\fint_{B_{\frac{3R}{4}}(x)}\vert Du\vert\operatorname{d}\!\xi \\
&+&c r^{\alpha}\left\lbrace \left[ M_{1-\alpha i_g,\frac{3R}{4}}(\mu)(x)\right] ^{\frac{1}{i_g}}+ \left[ \overline{M}_{1-\alpha i_g,\frac{3R}{4}}(\psi)(x)\right] ^{\frac{1}{i_g}}\right\rbrace \\
&+&cr^{\alpha} \left\lbrace W_{\frac{1}{i_g+1},i_g+1}^{\mu}(x,2R)+W_{ \frac{1}{i_g+1},i_g+1}^{[\psi]}(x,2R)\right\rbrace  \\
&+&cr^{\alpha} \int_{0}^{2R}[\omega(\rho)]^{\frac{1}{1+s_g}}G^{-1}\left[\fint_{B_{\rho}(x)}[G(|D\psi|)+G(|\psi|)]d\xi \right]\frac{d\rho}{\rho^{1+\alpha}}.
\end{eqnarray*}
Moreover, due to \eqref{mwanguji}, we have
\begin{eqnarray*}
\nonumber \fint_{B_r(x)}|Du|d\xi &\leqslant& c\left[ \fint_{B_{\frac{3R}{4}}(x)}|Du|d\xi+W_{\frac{1}{i_g+1},i_g+1}^{\mu}(x,\frac{3R}{2})+W_{ \frac{1}{i_g+1},i_g+1}^{[\psi]}(x,\frac{3R}{2})\right]  \\ \nonumber
&+&c\int_{0}^{\frac{3R}{2}}[\omega(\rho)]^{\frac{1}{1+s_g}}G^{-1}\left[\fint_{B_{\rho}(x)}[G(|D\psi|)+G(|\psi|)]d\xi \right]\frac{d\rho}{\rho},
\end{eqnarray*}
\begin{eqnarray*}
\nonumber \fint_{B_r(y)}|Du|d\xi &\leqslant& c\left[ \fint_{B_{\frac{3R}{4}}(y)}|Du|d\xi+W_{\frac{1}{i_g+1},i_g+1}^{\mu}(y,\frac{3R}{2})+W_{ \frac{1}{i_g+1},i_g+1}^{[\psi]}(y,\frac{3R}{2})\right]  \\ \nonumber
&+&c\int_{0}^{\frac{3R}{2}}[\omega(\rho)]^{\frac{1}{1+s_g}}G^{-1}\left[\fint_{B_{\rho}(y)}[G(|D\psi|)+G(|\psi|)]d\xi \right]\frac{d\rho}{\rho}.
\end{eqnarray*}
Because of  the above  estimates, we derive
\begin{eqnarray*}
|Du(x)-Du(y)|&\leqslant &c\left( \frac{r}{R}\right) ^{\alpha}\fint_{B_R(x_0)}\vert Du\vert\operatorname{d}\!\xi
\\&+&c r^{\alpha}\left\lbrace \left[ M_{1-\alpha i_g,\frac{3R}{4}}(\mu)(x)\right] ^{\frac{1}{i_g}}+ \left[ \overline{M}_{1-\alpha i_g,\frac{3R}{4}}(\psi)(x)\right] ^{\frac{1}{i_g}}\right\rbrace \\
&+&c r^{\alpha}\left[  W^{\mu}_{-\alpha+\frac{1+\alpha}{1+i_g},i_g+1}(x,2R)+ W^{[\psi]}_{-\alpha+\frac{1+\alpha}{1+i_g},i_g+1}(x,2R)\right] \\
&+&c r^{\alpha}\left[  W^{\mu}_{-\alpha+\frac{1+\alpha}{1+i_g},i_g+1}(y,2R)+ W^{[\psi]}_{-\alpha+\frac{1+\alpha}{1+i_g},i_g+1}(y,2R)\right] \\
&+&c r^{\alpha}\left[ \int_{0}^{2R}[\omega(\rho)]^{\frac{1}{1+s_g}}G^{-1}\left[\fint_{B_{\rho}(x)}[G(|D\psi|)+G(|\psi|)]d\xi  \right]\frac{d\rho}{\rho^{1+\alpha}} \right] \\
&+&c r^{\alpha}\left[  \int_{0}^{2R}[\omega(\rho)]^{\frac{1}{1+s_g}}G^{-1}\left[\fint_{B_{\rho}(y)}[G(|D\psi|)+G(|\psi|)]d\xi  \right]\frac{d\rho}{\rho^{1+\alpha}}\right],
\end{eqnarray*}
where we used the fact that $W_{\frac{1}{i_g+1},i_g+1}^{\mu}(x,2R) \leqslant c W^{\mu}_{-\alpha+\frac{1+\alpha}{1+i_g},i_g+1}(x,2R)$ and for other case there are similar inequalities.

For every $ \varepsilon>0 $, we know that there exists $ 0<r\leqslant R $ such that
\begin{equation*}
M_{1-\alpha i_g,\frac{3R}{4}}(\mu)(x)\leqslant |B_1|^{-1}\frac{|\mu|(B_{\frac{3r}{4}}(x))}{\left(\frac{3r}{4} \right) ^{n-1+\alpha i_g}}+\varepsilon.
\end{equation*}

We keep in mind  the definition of Wolff potential  and we derive
\begin{eqnarray*}
\frac{|\mu|(B_{\frac{3r}{4}}(x))}{\left(\frac{3r}{4} \right) ^{n-1+\alpha i_g}}&=&
\left[\left(\frac{|\mu|(B_{\frac{3r}{4}}(x))}{\left(\frac{3r}{4} \right) ^{n-1+\alpha i_g}} \right) ^{\frac{1}{i_g}} \frac{1}{-\log(3/4)}\int_{3r/4}^r\frac{\operatorname{d}\!\rho}{\rho}\right] ^{i_g}  \\
&\leqslant& C\left[ \int_{3r/4}^r\left(\frac{|\mu|(B_{\rho}(x))}{{\rho }^{n-1+\alpha i_g}} \right) ^{\frac{1}{i_g}} \frac{\operatorname{d}\!\rho}{\rho}\right] ^{i_g}  \\
&\leqslant& C\left[W^{\mu}_{-\alpha+\frac{1+\alpha}{1+i_g},i_g+1}(x,R) \right] ^{i_g}.
\end{eqnarray*}
Likewise,
\begin{equation*}
\frac{D\Psi(B_{\frac{3r}{4}}(x))}{\left(\frac{3r}{4} \right) ^{n-1+\alpha i_g}} \leqslant C\left[W^{[\psi]}_{-\alpha+\frac{1+\alpha}{1+i_g},i_g+1}(x,R) \right] ^{i_g}.
\end{equation*}
Finally, we take in account the definition of $r$ to obtain
\begin{eqnarray*}
&&\vert Du(x)-Du(y) \vert \\
&\leq &c\fint_{B_R(x_0)}\vert Du\vert\operatorname{d}\!\xi\left(\frac{|x-y|}{R} \right) ^{\alpha}  \\
&+& c \left[  W^{\mu}_{-\alpha+\frac{1+\alpha}{1+i_g},i_g+1}(x,2R)+ W^{[\psi]}_{-\alpha+\frac{1+\alpha}{1+i_g},i_g+1}(x,2R)\right]|x-y|^{\alpha} \\
&+&c \left[  W^{\mu}_{-\alpha+\frac{1+\alpha}{1+i_g},i_g+1}(y,2R)+ W^{[\psi]}_{-\alpha+\frac{1+\alpha}{1+i_g},i_g+1}(y,2R)\right]|x-y|^{\alpha} \\
&+&c \left[ \int_{0}^{2R}[\omega(\rho)]^{\frac{1}{1+s_g}}G^{-1}\left[\fint_{B_{\rho}(x)}[G(|D\psi|)+G(|\psi|)]d\xi  \right]\frac{d\rho}{\rho^{1+\alpha}} \right]|x-y|^{\alpha} \\
&+&c \left[  \int_{0}^{2R}[\omega(\rho)]^{\frac{1}{1+s_g}}G^{-1}\left[\fint_{B_{\rho}(y)}[G(|D\psi|)+G(|\psi|)]d\xi  \right]\frac{d\rho}{\rho^{1+\alpha}}\right]|x-y|^{\alpha},
\end{eqnarray*}
Then we finish the proof of Theorem \ref{Th2}.
\end{proof}

\section*{Acknowledgments}The authors are supported by the National Natural Science Foundation of China (NNSF  Grant No.12071229 and No.11671414).The authors would like to express their gratitude to the anonymous reviewers for their constructive comments and suggestions that improved the last version of the manuscript.


\begin{thebibliography}{99}
\bibitem{a1}\label{a1} R.A. Adams, Sobolev Spaces. Academic Press,  New York (1975)
\bibitem{b13}\label{b13} P. Baroni, Riesz potential estimates for a general class of quasilinear equations, Calc. Var. Partial Differential Equations 53 (3-4) (2015) 803-846.
\bibitem{bm20}\label{bm20} L. Beck, G. Mingione: Lipschitz bounds and nonuniform ellipticity, Comm. Pure Appl. Math. 73 (2020)  944-1034.
\bibitem{cho1}\label{cho1} Y. Cho,  Global gradient estimates for divergence-type elliptic problems involving general nonlinear
operators, J. Differ. Equ. 264(2018) 6152-6190.
\bibitem{cf1}\label{cf1}A. Cianchi, N. Fusco, Gradient regularity for minimizers under general growth conditions. J. Reine
Angew. Math. 507(1999) 15-36.
\bibitem{cm16}\label{cm16} A. Cianchi,  V. Maz'ya, Gradient regularity via rearrangements for p-Laplacian type elliptic boundary value problems, J. Eur. Math. Soc. (JEMS) 16 (2014) 571-595.
\bibitem{cm17}\label{cm17} A. Cianchi,  V. Maz'ya, Global Lipschitz regularity for a class of quasilinear elliptic equations,
Comm. Partial Differential Equations 36 (2011) 100-133.
\bibitem{cm15}\label{cm15} A. Cianchi,  V. Maz'ya, Global boundedness of the gradient for a class of nonlinear elliptic systems,
Arch. Ration. Mech. Anal. 212 (2014) 129-177.
\bibitem{de1}\label{de1}L. Diening, F. Ettwein, Fractional estimates for non-differentiable elliptic systems with general growth, Forum Math,  20(3) (2008) 523-556.
\bibitem{dm10}\label{dm10} F. Duzaar,  G. Mingione, Gradient estimates via non-linear potentials, Amer. J. Math. 133 (2011) 1093-1149.
\bibitem{dm11}\label{dm11} F. Duzaar,  G. Mingione, Gradient estimates via linear and nonlinear potentials, J. Funct. Anal. 259 (2010)  2961-2998.
\bibitem{g11}\label{g11} E. Giusti,  Direct Methods in the Calculus of Variations, World Scientific Publishing Co.,Inc.,River Edge (2003).
\bibitem{hphp1}\label{hphp1}P. Harjulehto, P. H{\"a}st{\"o}, Orlicz spaces and generalized Orlicz spaces,  Lecture Notes in Mathematics, 2236. Springer, Cham, 2019.
\bibitem{hw32}\label{hw32}  L. Hedberg, Th.H. Wolff, Thin sets in Nonlinear Potential Theory,  Ann. Inst. Fourier (Grenoble) 33 (1983) 161-187.
\bibitem{km5}\label{km5} T. Kilpel{\"a}inen,  J. Mal\'y,   The Wiener test and potential estimates for quasilinear elliptic equations,
Acta Math. 172 (1994) 137-161.
\bibitem{km6}\label{km6} T. Kilpel{\"a}inen,  J. Mal\'y, Degenerate elliptic equations with measure data and nonlinear potentials,  Ann. Scuola Norm. Sup. Pisa Cl. Sci. (4) 19 (1992)  591-613.
\bibitem{km12}\label{km12} T. Kuusi,  G. Mingione,  Universal potential estimates, J. Funct. Anal. 262 (2012), 4205-4269.
\bibitem{km13}\label{km13} T. Kuusi,  G. Mingione, Vectorial nonlinear potential theory, J. Eur. Math. Soc. 20 (2018)  929-1004.
\bibitem{km15}\label{km15} T. Kuusi,  G. Mingione, Guide to nonlinear potential estimates (English summary), Bull. Math. Sci. 4 (2014) 1-82.
\bibitem{km14}\label{km14} T. Kuusi,  G. Mingione,Y.  Sire, Nonlocal equations with measure data, Comm. Math. Phys. 337 (2015) 1317-1368.
\bibitem{l1}\label{l1} G.M.  Lieberman, The natural generalization of the natural conditions of Ladyzhenskaya and Ural'tseva for elliptic equations, Comm. Partial Differential Equations 16 (1991)  311-361.
\bibitem{l2}\label{l2} G.M.  Lieberman,  Regularity of solutions to some degenerate double obstacle problems,  Indiana Univ. Math. J. 40 (1991)  1009-1028.
\bibitem{m9}\label{m9} G. Mingione, Gradient potential estimates, J. Eur. Math. Soc. (JEMS) 13 (2011), 459-486.
\bibitem{mh31}\label{mh31}V.G.  Maz'ja, V.P.  Havin, A nonlinear potential theory, Uspehi Mat. Nauk 27 (1972) 67-138.
\bibitem{mz1}\label{mz1} L. Ma, Z.  Zhang,  Wolff type potential estimates for stationary Stokes systems with Dini-BMO coefficients, Commun. Contemp. Math., 2020, DOI: 10.1142/S0219199720500649.
\bibitem{rr28}\label{rr28} M.M. Rao, Z.D. Ren, Theory of Orlicz spaces, Monographs and Textbooks in Pure and Applied Mathematics, 146. Marcel Dekker, Inc., New York, 1991.
\bibitem{rt1}\label{rt1} J.F. Rodrigues, R. Teymurazyan, On the two obstacles problem in Orlicz-Sobolev spaces and applications (English summary),
Complex Var. Elliptic Equ. 56 (2011) 769-787.
\bibitem{s26}\label{s26} C. Scheven, Gradient potential estimates in non-linear elliptic obstacle problems with measure data,  J. Funct. Anal. 262 (2012)  2777-2832.
\bibitem{s28}\label{s28}  C. Scheven, Elliptic obstacle problems with measure data: potentials and low order regularity, Publ. Mat. 56 (2012)  327-374.
\bibitem{tw7}\label{tw7} N. Trudinger,  X. Wang, On the weak continuity of elliptic operators and applications to potential theory,
Amer. J. Math. 124 (2002)  369-410.
\bibitem{tw8}\label{tw8} N. Trudinger,  X. Wang, Quasilinear elliptic equations with signed measure data,  Discrete Contin. Dyn. Syst. 23 (2009) 477-494.
\bibitem{xiong1}\label{xiong1}Q. Xiong,  Z.  Zhang,  Gradient potential estimates for elliptic obstacle problems,  J. Math. Anal. Appl. 495(2021) 124698.
 \bibitem{yz1}\label{yz1}F.P. Yao, S.  Zhou, Calder$\acute{o}$n-Zygmund estimates for a class of quasilinear elliptic equations, J. Funct. Anal. 272 (2017), 1524-1552.







\end{thebibliography}
\end{document}